\documentclass[12pt]{article}

\usepackage{latexsym,amssymb,amsmath,bm}

\newcommand{\C}{\mathbb C}
\newcommand{\R}{\mathbb R}
\newcommand{\Z}{\mathbb Z}
\newcommand{\Q}{\mathbb Q}
\newcommand{\N}{\mathbb N}
\newcommand{\F}{\mathbb F}
\newcommand{\E}{\mathbb E}
\newcommand{\A}{\mathbb A}
\newcommand{\PP}{\mathbb P}
\newcommand{\slz}{SL(2,\mathbb Z)}
\newcommand{\glz}{GL(2,\mathbb Z)}
\newcommand{\slw}{SL(2,\mathbb C)}
\newcommand{\glc}{GL(3,\mathbb C)}

\newcommand{\glt}{GL(3,\mathbb F)}
\newcommand{\sltf}{SL(2,\mathbb F)}
\newcommand{\gltf}{GL(2,\mathbb F)}
\newcommand{\glw}{GL(2,\mathbb C)}
\newcommand{\sma}{\left(\begin{array}}
\newcommand{\fma}{\end{array}\right)}
\newcommand{\injects}{\hookrightarrow}

\newtheorem{lem}{Lemma}[section]

\newtheorem{co}[lem]{Corollary}
\newtheorem{thm}[lem]{Theorem}
\newtheorem{prop}[lem]{Proposition}

\newtheorem{qu}[lem]{Question}
\newenvironment{proof}{\textbf{Proof.}}{\newline\hspace*{\fill}{$\Box$}\\}

\begin{document}
\title{Minimal dimension faithful linear representations of common
finitely presented groups}
\author{J.\,O.\,Button}

\newcommand{\Address}{{% additional braces for segregating \footnotesize
  \bigskip
  \footnotesize

\textsc{Selwyn College, University of Cambridge,
Cambridge CB3 9DQ, UK}\par\nopagebreak
  \textit{E-mail address}: \texttt{j.o.button@dpmms.cam.ac.uk}
}}

\date{}
\maketitle
\begin{abstract}
For various finitely presented groups, including right angled Artin groups
and free by cyclic groups, we investigate what is the smallest dimension
of a faithful linear representation. This is done both over $\C$ and over
fields of positive characteristic. In particular we show that Gersten's
free by cyclic group has no faithful linear representation of dimension
4 or less over $\C$, but has no faithful linear representation of any
dimension over fields of positive characteristic. 
\end{abstract}

\section{Introduction}

A common question to ask about a given finitely presented group is whether
it is linear. We have both algebraic and geometric arguments available
for proving linearity, for instance \cite{nis} and \cite{sha} in the first
case, whereas a group which is seen to be the fundamental group of a
finite volume hyperbolic orientable 3-manifold will embed in $\slw$.
There are also arguments for showing a group is not linear, with lack
of residual finiteness surely the most utilised. This means that our options
decrease drastically when trying to show that a residually finite group
is not linear.

However even if we know a group is linear, this still begs the question
of what dimension a faithful linear representation will be, and in
particular what is the minimum dimension of a faithful linear representation
of a given group. This is a question that has been considered for finite groups,
for instance \cite{babnn} which applies it to results on expansion of finite
groups (in fact the quantity used is the minimum dimension of a non trivial
linear representation, but this is the same as a faithful linear
representation for finite simple groups, which are the main examples
studied). 
Now it is clear that if an infinite group $G$
contains a finite subgroup $H$ having no low dimensional faithful linear
representation then the same is true of $G$, thus we obtain an 
potential obstruction. However nearly all of the finitely presented
groups considered in this paper will be torsion free, so this
observation will not help and we will need to establish our own
criteria for showing the non-existence of low dimensional faithful
representations, as well as producing existence results which will
generally be constructive.

One motivation for studying this question is that although linearity 
in general ought
to provide insight into the structure of a group, it is clear that a
faithful representation of low dimension will be a lot more useful in
practice, for instance computationally, than one which uses matrices of
vast size. Conversely proving that a given group has no low dimensional 
faithful linear representations can be seen as indicating that this group
has some measure of complexity in its internal structure.

So far we have been using the phrase ``low dimension'' without quantification,
however in this paper the term will mean dimension 4 or less. This fits in
with an old question credited to Thurston, asking if every finitely
presented 3-manifold group has a faithful 4-dimensional linear representation,
which we showed to be false in \cite{me3}. Also for matrices of this
size it is sometimes possible to provide direct hands on arguments for various
groups on the existence or not of a low dimensional faithful representation.
This will often entail looking at the centraliser of a particular element
and considering the possible Jordan normal forms.

We have also not yet specified our definition of a linear group, which we
would expect in theory to be a subgroup of $GL(d,\F)$ for some dimension
$d$ and {\it some} field $\F$, but in practice we are likely to be
thinking about the case when $\F$ is the complex numbers. In fact our
viewpoint will be the more general one. Indeed once it is pointed out that
for finitely generated groups, the existence of an embedding in
$GL(d,\F)$ for $\F$ any field of characteristic zero is equivalent to
the existence of an embedding in $GL(d,\C)$, it makes sense to group
together different fields with the same characteristic. Thus for a
finitely generated or presented group $G$ and a characteristic $p$
equal to zero or a prime, we will define $m_p(G)$ to be the minimum
dimension of a faithful linear representation of $G$ over some field
of characteristic $p$, with $m_0(G)$ then reducing to the case of $\C$.
This raises the possibility that the values of $m_p(G)$ could vary
with $p$, or at least could be different for $p=0$ and prime $p$,
although we will want as far as possible to come up with
existence and non-existence results which are independent of the 
characteristic. This we are able to do initially, although it becomes
harder as we move to more complicated examples.

We start in the next section with some basic observations (these are
independent of the characteristic); in particular we have an
obstruction for the existence of a faithful 3-dimensional embedding in 
Theorem \ref{5cm}, as well as quoting some results in the literature
that do establish linearity for (some amalgamated) free products. It
is well known that free groups and free abelian groups admit 2-dimensional
faithful representations, so Section 3 considers the obvious generalisation
of these groups, namely the much studied right angled Artin groups (RAAGs).
These are known to be linear (even over $\Z$ although we will not be
considering those particular representations here), at least in characteristic
zero. Although we do not know of a definitive result as to which 
finitely presented groups have faithful 2-dimensional representations,
it is straightforward to do this in Proposition \ref{2d}
for RAAGs in terms of their defining graphs. Moreover our results provide 
some evidence towards the equivalent of Thurston's question being true
for RAAGs in that, although we show in Section 2 that the direct product
$F_2\times F_2$ of rank 2 free groups requires at least 4 dimensions
for faithful representations, Corollary \ref{mainc} states that
the path with 4 vertices embeds into
$GL(3,\F)$ for some field $\F$ of any given characteristic. By a result
in \cite{kmkb} this means that all RAAGs whose defining graph is a tree
will also embed in $GL(3,\F)$.   

Section 4 is mainly a review of known results on the minimum 
dimension of faithful linear representations for groups commonly
occurring in geometry, such as braid groups, automorphism groups,
3-manifold groups and word hyperbolic groups.

For the rest of the paper we consider free by cyclic groups, which are 
of the form $F_n\rtimes_\alpha\Z$ for $F_n$ the rank $n$ free group
and $\alpha$ an automorphism of $F_n$. In Section 5 we look at the
case $n=2$. Here it is know that all such groups
are linear but we find the minimum dimension
for all but one family of groups. In particular Proposition \ref{finite}
shows that if $\alpha$ has finite order in $Out(F_2)$ then the minimum
dimension is 2.

However for general $n$ the question of whether $F_n\rtimes_\alpha\Z$ is
always linear is unknown (this is Problem 5 in \cite{ds}). A particularly
good test case is the Gersten group with $n=3$, where for $F_3$ free on
$a,b,c$ the automorphism $\alpha$ fixes $a$, sends $b$ to $ba$ and $c$ to 
$ca^2$. This was introduced in \cite{ger} and shown to admit no proper,
cocompact action by isometries on a CAT(0) space (nor can it be a subgroup
of a group with such an action). In Section 7
we show that this group has no faithful linear representation of dimension
4 or less over any field, so Thurston's question also has a negative
answer for
free by cyclic groups as well as 3-manifold groups. However, perhaps
more interestingly and certainly surprisingly, we show in Section 6 
that this group has no faithful
linear representation in any dimension over any field of positive
characteristic. This makes the question of linearity of the Gersten
group over $\C$ a crucial one, for either it behaves very differently in
zero characteristic than in positive characteristic, or we have a free by
cyclic group which is not linear over any field.  

\section{Preliminaries}
Given a group $G$ we will write $m(G)$ for 
the minimum value of the dimension $d$ such that 
there is some field $\F$ where $G$
admits a faithful representation into 
$GL(d,\F)$. Here we do allow $\F$ to be any field,
although we will find it useful to distinguish between
different characteristics. Consequently we also define 
$m_p(G)$ to be the minimum dimension $d$ of a faithful representation
into $GL(d,\F)$ over all fields $\F$ of characteristic $p$, 
where $p$ is zero or a prime, thus $m(G)=\mbox{min}_p\,m_p(G)$.
It can happen that $G$ is known to
have no such faithful representation in any dimension, whereupon
we will write $m(G)=\infty$ and/or $m_p(G)=\infty$,
or the existence of such a faithful linear representation 
could even be an open question.

In this paper all groups considered will be abstractly
finitely generated. Thus on first turning to the
characteristic zero case, we can take $\F=\C$ without loss of generality.
This is because we can consider the coefficient field 
$\E=\Q(x_1,\ldots ,x_n)$ where $x_1,\ldots ,x_n$ are all entries taken
from a finite generating set. As $\E$ is a finitely generated field
over $\Q$, it must be a finite extension of a purely transcendental
field of finite transcendence degree. But this latter field will
embed in $\C$ just by taking the right number of transcendentally
independent elements and then $\E$ will too because $\C$ is
algebraically closed. (Tao has referred to this in his blog as the
``Baby Lefschetz Principle''.) As for positive characteristic, although
$G$ would be a finite group if we take $\F$ to be the finite field
$\F_p$, or even if $\F$ is a locally finite field such as the
algebraic closure $\overline{\F_p}$ because $G$ is finitely generated,
we shall be taking fields
such as $\F_p(x_1,\ldots ,x_n)$ for $x_1,\ldots ,x_n$ transcendentally
independent elements or some finite extension of this. Indeed we can
invoke a similar ``Baby Lefschetz Principle'' by working over a
universal coordinate domain in prime characteristic $p$, that is an
algebraically closed field of infinite transcendence degree over
$\F_p$. 

In this paper we are interested in faithful linear representations
of low dimension, which here usually means 4 or less. Moreover we will be
considering groups that are finitely generated, indeed finitely
presented, and nearly always torsion free. To this end we first make some
well known background comments.   
In the 1-dimensional case, we merely note that $GL(1,\F)=\F^*$ is abelian
for any field $\F$
and that $\Z^n$ embeds in $\C^*$ for any $n$, indeed in $\F^*$ for
$\F$ some field of characteristic $p$ because
$\Z^n\cong\langle x_1,\ldots ,x_n\rangle\leq \F_p(x_1,\ldots ,x_n)^*$.
For linearity in
2 dimensions, a well known necessary condition
is that a subgroup $H$ of $SL(2,\F)$
not containing the element $-I_2$ has to be commutative transitive.
This does not quite hold in $GL(2,\F)$ but we can rephrase it
here as: if $H\leq GL(2,\F)$ but there exists $x,y,z\in H$ where $x$
commutes with $y$ and $y$ with $z$ but $xz\neq zx$ then $y=\lambda I_2$
for some scalar $\lambda\in\F^*$
and in particular $y$ must be in the centre of $H$. Well known finitely
generated groups
that embed in $GL(2,\C)$ are free groups $F_r$ of any rank $r$, closed
orientable surface groups $\pi_1(S_g)$ of any genus $g$,
the fundamental groups $\pi_1(M^3)$ of closed orientable hyperbolic
3-manifolds, 
all limit groups (thus generalising $F_r$ and $\pi_1(S_g)$),
and the fundamental group of a graph
of groups where the graph is a tree, the vertex groups are finite rank free, 
and the edge groups are maximal cyclic subgroups (by \cite{sha} Theorem 2).

As for faithful representations of these groups over fields of
positive characteristic, much less seems to be known. We certainly
have that $m_p(F_r)=2$ for $r\geq 2$ and
any prime $p$, for instance by \cite{wehbk} Exercise 2.2.
This then allows us to conclude (see for instance
\cite{lubseg} Window 8 Theorem 1) 
that $m_p(G)=2$ for any prime $p$ and any non-abelian limit
group $G$. Also we shall see in subsection 2.2 that \cite{sha} Theorem 2
goes through in positive characteristic, hence we again have $m_p(G)=2$
for the fundamental group of those graph of groups 
described above. However we
do not know of a general result for $\pi_1(M^3)$, or even that
$m_p(\pi_1(M^3))$ is always finite.

\subsection{In three dimensions}  
Moving now to three dimensional representations,
one consequence of the limited number of Jordan normal forms for 3 by
3 matrices is the following:
\begin{prop} \label{jnf}
If $\F$ is algebraically closed and
the centraliser of $A\in\glt$ is not soluble then $A$ is
diagonalisable with a repeated eigenvalue.
\end{prop}
\begin{proof} It is easily checked that the centraliser in $\glt$ of
\[A=\sma{ccc}\lambda&1&0\\0&\lambda&1\\0&0&\lambda\fma\mbox{ for }
\lambda\neq 0\mbox{ or }
\sma{ccc}\lambda&0&0\\0&\mu&0\\0&0&\nu\fma\]
for distinct $\lambda,\mu,\nu\neq 0$ is abelian. If
\[A=\sma{ccc}\lambda&0&1\\0&\lambda&0\\0&0&\lambda\fma\mbox{ for }
\lambda\neq 0\]
then its centraliser $C(A)$ in $GL(3,\F)$
consists of elements of the form
$\sma{ccc}a&b&c\\0&d&e\\0&0&a\fma$ for $a,d\neq 0$
and the commutator of any two
such elements is an upper unitriangular matrix, hence the derived
subgroup of $C(A)$ is nilpotent. Similarly if
$A=\sma{ccc}\lambda&1&0\\0&\lambda&0\\0&0&\mu\fma$
then the centraliser also consists of upper triangular
matrices.  As $\F$ is algebraically closed, any element is conjugate
to one in Jordan normal form and these (or a similar matrix) have all
been considered apart from diagonal matrices with a repeated eigenvalue.
\end{proof}

As an immediate corollary, we obtain
\begin{co} \label{jnfc} Suppose that a group $G$ contains commuting elements
$A$ and $B$ which have centralisers $C_G(A), C_G(B)\neq G$ that are both 
non soluble and which are not equal. 
Then if there is a faithful representation
of $G$ into $\glt$ for $\F$ any algebraically closed field,
we can conjugate the image in $\glt$ such that
\[A=\sma{ccc}\lambda&0&0\\0&\lambda&0\\0&0&\mu\fma
\mbox{ and }B=\sma{ccc}x&0&0\\0&y&0\\0&0&x\fma\]
for some non zero scalars $\lambda\neq\mu$, $x\neq y$.
\end{co}
\begin{proof}
By Proposition \ref{jnf} we know $A$ and $B$ are both diagonalisable
with a repeated eigenvalue and as they commute they must be
simultaneously diagonalisable. But $A$ and $B$ are not in the centre and
the centralisers are distinct so they both have exactly two distinct
eigenvalues, with the repeat appearing in a different place for $A$
and for $B$. Without loss of generality this gives the form above.
\end{proof}

This leads to the following result which we can use
as an obstruction for a faithful 3-dimensional embedding.
\begin{thm} \label{5cm}
Suppose a group $G$ has elements $A,B,C,D$ where $A$ commutes with
$B$, $B$ with $C$, $C$ with $D$ and $D$ with $A$. If both
$\langle A,C\rangle$ and $\langle B,D\rangle$ are not soluble then
$G$ cannot embed in $\glt$ for any field $\F$. 
\end{thm}
\begin{proof}
We see that $A$ does not commute with $C$ and $B$ does not with $D$, but
if the centralisers
$C_G(A)$ and $C_G(B)$ were equal then $\langle A,B,C,D\rangle$
would be abelian. Thus on taking the algebraic closure 
$\overline{\F}$ of $\F$ and using
Corollary \ref{jnfc}, we get that 
$A$ and $B$ are in the form above. Hence we must have
\[C=\sma{ccc}\alpha&0&\beta\\0&\eta&0\\
\gamma&0&\delta\fma \mbox{ and }
D=\sma{ccc}a&b&0\\c&d&0\\0&0&e\fma.\]
But on equating $CD=DC$ we find that $c\beta=0$. Now if
$c=0$ then $D$ is upper triangular which makes $\langle B,D\rangle$
soluble, whereas if $\beta=0$ the same is true of $\langle A,C\rangle$.
Thus $\langle A,B,C,D\rangle$ does not embed in 
$GL(3,\overline{\F})$ and so nor does $G$.
\end{proof} 

This immediately 
gives the following Corollary (possibly well known but we have never
seen it written down).
\begin{co} \label{nosq}
The direct product $F_2\times F_2$ (or any group containing $F_2\times F_2)$
cannot embed in $\glt$ for $\F$ any field.
\end{co}
Of course faithful representations of $F_2\times F_2$ abound in $GL(4,\F)$
for suitable fields $\F$ of any characteristic
by taking the direct product of two 2-dimensional representations of $F_2$,
so that we have $m(F_2\times F_2)=4=m_p(F_2\times F_2)$ for any
characteristic $p$.

\subsection{Establishing linearity}
We finish this section by quoting some results that do establish
linearity. It was shown in \cite{nis} by Nisnevi\u{c}
that the free product of
linear groups is linear. More precisely, 
if $A$ and $B$ are subgroups of
$GL(d,\F)$ for $\F$ a field of characteristic $p$ (zero or prime)
then $A*B$ embeds in $GL(d+1,\F')$ where $\F'$ is some field having the
same characteristic as $\F$. Moreover if neither
$A$ nor $B$ contain any scalar matrices
except the identity then $A*B$ embeds in $GL(d,\F')$. Further
results are in \cite{weh} and \cite{sha}, with the latter giving
a useful general result to establish the linearity of free
products with an abelian amalgamated subgroup, which we state here as
\begin{prop} (\cite{sha} Proposition 1.3) \label{shl}
Suppose $G_1*_HG_2$ is a free product with abelian amalgamated
subgroup $H$ and suppose we have faithful representations
$\rho_i:G_i\injects GL(d,\F)$ for $d\geq 2$ and
$i=1,2$ over any field $\F$ such that\\
(a) $\rho_1$ and $\rho_2$ agree on $H$,\\
(b) $\rho_1(h)$ is diagonal for all $h\in H$ and\\
(c) For all $g\in G_1\setminus H$ we have that the bottom left hand entry
of $\rho_1(g)$ is non zero, and similarly for the top right hand entry
of $\rho_2(g)$ for all $g\in G_2\setminus H$.

Then $G_1*_HG_2$ embeds in $GL(d,\F(t))$ where $t$ is a transcendental
element over $\F$. In particular, if $\F=\C$ or our universal domain
of positive
characteristic $p$ then $G_1*_HG_2$ also embeds in
$GL(n,\F)$ as here $\F(t)$ embeds into $\F$.
\end{prop}
\begin{proof} In \cite{sha} the result was stated just for $\C$ (and for
$SL(d,\C)$ rather than $GL(d,\C)$) but the proof goes through in general
so here we just give a summary.

Define the representation $\rho:G_1*_HG_2\rightarrow GL(d,\F(t))$ as
equal to $\rho_2$ on $G_2$ but on $G_1$ we replace $\rho_1$ by the
conjugate representation $T\rho_1T^{-1}$ where $T$ is the diagonal
matrix $\mbox{diag}(t,t^2,\ldots ,t^n)$, and then extend to all of
$G_1*_HG_2$. Now it can be shown straightforwardly 
that any element not conjugate into $G_1\cup G_2$ is
conjugate in $G_1*_HG_2$ to something with normal form
\[g=\gamma_1\delta_1\ldots \gamma_l\delta_l\]
where all $\gamma_i\in G_1\setminus H$ and all $\delta_i\in G_2\setminus H$.
Induction on $l$ then yields that the entries of $g$ are Laurent polynomials
in $t^{\pm 1}$ with coefficients in $\F$ and with the bottom right 
hand entry of $g$ equal to $\alpha t^{l(d-1)}+\ldots$ 
where all other terms are of
strictly lower degree in $t$. But it can be checked
that $\alpha$ is actually just a product of these respective
bottom left and top right entries, thus is a non zero element of $\F$
so this bottom right hand entry does not equal 1 and
$g$ is not the identity matrix.
\end{proof}   

We end with a mention of direct products: if $G=A\times B$ then we can take
a direct sum of faithful representations for $A$ and $B$ to get
$m_p(G)\leq m_p(A)+m_p(B)$ for any $p$. Equality can be
obtained even for torsion free groups, such as $A=B=F_2$ but an easy lemma
allows us to do better in one case.
\begin{lem} \label{dp} For any characteristic $p$, any $n\in\N$ and
any group $G$ we have $m_p(G)=m_p(G\times \Z^n)$.
\end{lem}
\begin{proof} We proceed by induction on $n$, so assume $n=1$. Given our
faithful representation in dimension $d$ of $G$ over the field $\F$, which
we can assume is the coefficient field of $G$, then on taking a number $t$
which is transcendental over $\F$ we add the matrix $tI_d$ to $G$.
This matrix has infinite order and no positive power of it is equal
to an element of $G$. We now replace $\F$ by $\F(t)$ and repeat $n$ times.
\end{proof}  
No such straightforward result can exist for semidirect products
$G\rtimes\Z^n$, even if $n=1$, as will be seen in Sections 5, 6 and 7.

\section{Right angled Artin groups}

Right angled Artin groups (RAAGs) have been much studied, especially
in recent years. We assume here that the underlying graph $\Gamma$ is
finite with $n$ vertices, in which case this RAAG
denoted by $R(\Gamma)$ embeds in a right angled Coxeter group
with $2n$ vertices by \cite{hsuw} and so is a subgroup of $GL(2n,\C)$,
using the standard faithful linear representation of Coxeter groups,
or even $GL(2n,\Z)$.
Thus we have that $m_0(R(\Gamma))\leq 2n$ 
and we are interested in obtaining better bounds or even the
exact value of $m$ or $m_0$ for various examples. That this is a non trivial
question is suggested by the analogous question on when a RAAG
$R(\Gamma)$ embeds in another RAAG $R(\Delta)$, with 
$m_0(R(\Gamma))\leq
m_0(R(\Delta))$ being an obvious necessary condition. Amongst the many
papers on this intriguing problem, we mention \cite{kim}, \cite{kam},
\cite{kmkb} and \cite{crdk}.

We start with the smallest dimensional representations and will consider 
both zero and positive characteristic, as here our results work equally well 
in all cases.
Obviously $m(R(\Gamma))=1$ (equivalently $m_p(R(\Gamma))=1$ 
for any characteristic $p$)
if and only if $R(\Gamma)$ is abelian if and
only if $\Gamma$ is a complete graph. We now consider the two dimensional
case.
\begin{prop} \label{2d} We have $m(R(\Gamma))=2$ 
(equivalently $m_p(R(\Gamma))=2$ for any characteristic $p$)
if and only if, on removal of all
vertices joining every other vertex, the resulting graph has at least
2 connected components, all of which are complete subgraphs.
\end{prop}
\begin{proof}
First if $\Gamma$ consists of two or more connected components which
are all complete subgraphs then $R(\Gamma)$ is a free product of free
abelian groups, so we can use the result \cite{nis}
of Nisnevi\u{c} mentioned in the
last section to get that $m(R(\Gamma))=m_p(R(\Gamma))=2$ for any
characteristic $p$ provided we avoid scalar matrices, which can be
done by starting with a 1-dimensional faithful representation
of a factor $\Z^n$ and extending this to a 2-dimensional representation by 
putting a single 1 in the bottom right hand corner below 
(ie take the direct sum with the trivial representation).
If there are
now also $k$ vertices joined to every other vertex (including each
other) then these elements will be in the centre and will provide
a direct factor of $\Z^k$, so we can apply Lemma \ref{dp}.

For the converse, we mentioned in the last section that a group which embeds 
in $GL(2,\F)$ must be very close to being commutative transitive. In
particular if $\Gamma$ is not complete then suppose first that
we have vertices
$v_1,v_2,v\in\Gamma$ with $v_1$ not joined to $v_2$ but $v$ joined to them
both. Then if $R(\Gamma)\leq GL(2,\F)$ we see that the element in
$R(\Gamma)$ corresponding to $v$ is in the centre and thus joined to every
other vertex. On removal of all such vertices (and all incident edges), which
will not affect the value of $m(R(\Gamma))$ or $m_p(R(\Gamma))$ by
Lemma \ref{dp}, we
will be left with a graph (or the empty set if $\Gamma$ was complete) with
no such $v_1,v_2,v$. This means that the relation of two vertices being
joined is an equivalence relation (at least if we regard a vertex as
related to itself) and so the connected components of $\Gamma$ must
be equivalence classes and so are complete subgraphs.
\end{proof}

However once we move above dimension 2 things become less obvious. 
At least we have immediately from Corollary \ref{nosq}:
\begin{thm} \label{rgsq}
If $m(R(\Gamma))\leq 3$ then $\Gamma$ does not have any
induced squares.
\end{thm}
\begin{proof}
On taking the subgroup of $R(\Gamma)$ generated by
the vertex elements of the induced square in such a faithful representation,
it is known for RAAGs that this will generate the RAAG given by a square
which is $F_2\times F_2$.
\end{proof}

We can now look through (or make) a list of small graphs and use the above
results and techniques to find the exact value of $m(R(\Gamma))\in\{1,2,3,4\}$
for many $\Gamma$. Starting with $\Gamma$ having at most 4 vertices (but
not complete), we can first use Lemma \ref{dp} to remove all vertices
joined to every other vertex and then see if we have a disjoint union
of complete subgraphs, thus providing those graphs $\Gamma$ 
with $m(R(\Gamma))=2$ by Proposition \ref{2d}. 
This works for all but the square, where $m(R(\Gamma))=4$ as discussed
earlier, and the two other RAAGs 
$\,\bullet{\bf -}\bullet{\bf -}{\bullet}\quad\bullet$
and $\,\bullet{\bf -}\bullet{\bf -}{\bullet}{\bf -}\bullet$.
The first is easily dealt with using the same techniques because a free 
product $G$ of those groups in Proposition
\ref{2d} will (if not itself such a group) have $m(G)=3=m_p(G)$
for any $p$. This follows as we can similarly
extend the 2-dimensional faithful representations
of the factors to 3-dimensional representations by putting a single 1
in the bottom right hand corner. We now have no scalar matrices
except the identity in any factor
so the free product result \cite{nis} applies.

This leaves 
$\,\stackrel{G}{\bullet}
{\bf -}\stackrel{Q}{\bullet}{\bf -}\stackrel{P}{\bullet}{\bf -}
\stackrel{F}{\bullet}$ which 
we now look in detail as it provides a useful
application of  Shalen's Proposition \ref{shl}, where we have given a label
to each generator of the RAAG.   First note by Corollary \ref{jnfc}, where
now $A=Q$ and $B=P$, that if a faithful representation in $\glt$ exists then
we can assume that $Q=\mbox{diag}(\lambda,\lambda,\mu)$ and
$P=\mbox{diag}(x,y,x)$. Consequently $F$ and $G$ must have the same form
as $C$ and $D$ respectively in the proof of Theorem \ref{5cm} before we
equate $CD=DC$. This gives rise to expressions for our matrices
$G,Q,P,F$ immediately prior to the application of Proposition \ref{shl}, 
which we denote by $G_0,Q_0,P_0,F_0$ respectively. In order to avoid surplus
variables which will turn out to be unnecessary,
we simplify $G_0,Q_0,P_0,F_0$ by setting $e=\eta=1$ and 
$x=\lambda,\mu=y=\lambda^{-1}$ which gives us:
\begin{thm} \label{main}
For the RAAG $\mathcal R$ defined by the graph
$\,\stackrel{G}{\bullet}
{\bf -}\stackrel{Q}{\bullet}{\bf -}\stackrel{P}{\bullet}{\bf -}
\stackrel{F}{\bullet}$,
consider the following matrices, either over the field $\F=\C$ with
prime subfield $\PP=\Q$ or over $\F$ 
a universal coefficient domain of characteristic $p$ with
prime subfield $\PP=\F_p$.
\[G_0=\sma{ccc}a&b&0\\c&d&0\\0&0&1\fma,
Q_0=\sma{ccc}\lambda&0&0\\0&\lambda&0\\0&0&\lambda^{-1}\fma,
P_0=\sma{ccc}\lambda&0&0\\0&\lambda^{-1}&0\\0&0&\lambda\fma,
F_0=\sma{ccc}\alpha&0&\beta\\0&1&0\\
\gamma&0&\delta\fma
\]
where $ad-bc,\alpha\delta-\beta\gamma,\lambda\neq 0$ so that 
$G_0,Q_0,P_0,F_0\in\glt$. 
Suppose the following conditions are satisfied:\\
(1) The matrix $\sma{cc}\alpha&\beta\\
\gamma&\delta\fma$ is such that no non trivial power has an off diagonal
entry which is zero.\\
(2) The matrices
$\sma{cc}a&b\\c&d\fma,
\sma{cc}\lambda&0\\0&\lambda^{-1}\fma\in GL(2,\F)$ generate a non abelian free
group such that any matrix in this group with top left hand entry equal
to 1 is the identity $I_2$.\\
Then on first conjugating these four matrices by $F_0$ (so that $G_0$ and
$Q_0$ change but $P_0$ and $F_0$ stay the same), then conjugating
just $P_0$ and $F_0$ by $D=\mbox{diag}(\phi,\phi^2,\phi^3)$, where
$\phi$ is a transcendental element over the field
$\PP(a,b,c,d,\alpha,\beta,\gamma,\delta,\lambda)$ (but now leaving
$G_0$ and $Q_0$ alone so that $P_0$ is still unchanged), we have that
the resulting matrices $G=F_0G_0F_0^{-1}$, $Q=F_0Q_0F_0^{-1}$, 
$P=DP_0D^{-1}=P_0$,
$F=DF_0D^{-1}$ provide a faithful representation of this RAAG 
$\,\stackrel{G}{\bullet}
{\bf -}\stackrel{Q}{\bullet}{\bf -}\stackrel{P}{\bullet}{\bf -}
\stackrel{F}{\bullet}$
in $\glt$.
\end{thm}
\begin{proof}
We apply Proposition \ref{shl} by regarding the RAAG 
$\,\stackrel{G}{\bullet}
{\bf -}\stackrel{Q}{\bullet}{\bf -}\stackrel{P}{\bullet}{\bf -}
\stackrel{F}{\bullet}$
as $\,\stackrel{G}{\bullet}
{\bf -}\stackrel{Q}{\bullet}{\bf -}\stackrel{P}{\bullet}
\,\cong F_2\times\Z$ amalgamated with 
$\,\stackrel{P}{\bullet}{\bf -}
\stackrel{F}{\bullet}
\,\cong\Z\times\Z$ over 
$\,\stackrel{P}{\bullet}\,\cong\Z$. 
First note that the group $\Gamma_1:=\langle P_0,F_0\rangle$
is a faithful representation of $\Z^2$ because Condition (1)
says that no non trivial power of $F_0$ is
diagonal. This also applies to $P_0^iF_0^j$ if $j\neq 0$, else we
could multiply by $P_0^{-i}$ and use the fact that pre or post multiplication
of a given matrix by any invertible diagonal matrix
does not alter which entries of the given matrix are
zero. In particular the only diagonal elements of 
$\Gamma_1$ are
$\langle P_0\rangle\cong\Z$ and indeed these are the only elements
of $\Gamma_1$ with a zero top right or bottom left hand entry. 

Now $\Gamma_2:=\langle G_0,P_0,Q_0\rangle\cong F_2\times\Z$ 
with $Q_0$ as the generating central element, by Condition
(2) and the fact that the only diagonal elements of
$\langle G_0,P_0\rangle$ are powers of $P_0$ (because any such element would
commute with $P_0$). As before this also implies that the only diagonal
elements of $\langle G_0,P_0,Q_0\rangle$ are $\langle P_0,Q_0\rangle\cong\Z^2$.
However we must now conjugate 
$\langle  G_0,Q_0,P_0\rangle$ in order to satisfy the zeros condition in
Proposition \ref{shl} and we use $F_0$ for this purpose. Although the
exact form of $F_0G_0F_0^{-1}$ is not very enlightening, it is easily
checked that conjugating a matrix of the form 
\[\sma{ccc}u&?&0\\ ?&?&0\\0&0&v\fma
\mbox{ by an invertible matrix of the form }
\sma{ccc}*&0&*\\0&*&0\\ *&0&*\fma,\]
where $u,v\in \F$ and $?$ is used to stand for arbitrary 
elements of $\F$ whereas $*$ denotes arbitrary non zero elements of $\F$,
results in
something with its top right or bottom left entry equal to zero if and
only if $u=v$.
Thus let us suppose there is an element $\gamma=w(G_0,P_0)Q_0^j\in\Gamma_2$ 
with its first and third diagonal entries equal to each other.
The latter entry must be equal to $\lambda^{i-j}$ where $i$ is the
exponent sum of $P_0$ in the word $w$. Thus $\gamma Q_0^{-j}$ is in
$\langle G_0,P_0\rangle\cong F_2$
with its first diagonal entry equal to
$\lambda^{i-2j}$ and third equalling $\lambda^i$. But now we can
consider $P_0^{2j-i}\gamma Q_0^{-j}$ with top diagonal entry equal to 1, 
which is also in $F_2$. By Condition (2) of this theorem, this element 
is the identity so $\gamma=P_0^{i-2j}Q_0^j$ has (equal) top/bottom diagonal 
entries $\lambda^{i-j}$ and $\lambda^{i-3j}$ respectively. This forces $j=0$ so
$\gamma$ was in $\langle P_0\rangle$ anyway.

Thus $\Gamma_1$ and $F_0\Gamma_2F_0^{-1}$ both
satisfy the property that the only
elements whose top right or bottom left entry is  non zero must lie in
the diagonal subgroup $\langle P_0\rangle\cong\Z$. We are now ready
to apply Proposition \ref{shl} so let us therefore conjugate $\Gamma_1$ by
$\mbox{diag}(\phi,\phi^2,\phi^3)$. 
For the proof to go through, we 
need that $\phi$ is transcendentally independent
of the coefficient field of $\langle \Gamma_1,F_0\Gamma_2F_0^{-1}\rangle
=\langle G_0,Q_0,P_0,F_0\rangle$ which is indeed how $\phi$ was chosen, thus
we conclude that $\langle F_0G_0F_0^{-1},F_0Q_0F_0^{-1},DP_0D^{-1},DF_0D^{-1}
\rangle$
generates a faithful representation in $\glt$ of the RAAG 
 $\,\stackrel{G}{\bullet}
{\bf -}\stackrel{Q}{\bullet}{\bf -}\stackrel{P}{\bullet}{\bf -}
\stackrel{F}{\bullet}$
where $G,Q,P,F$ are the natural generators of this RAAG and where we
are taking
$G=F_0G_0F_0^{-1}$, $F=F_0Q_0F_0^{-1}$, $P=DP_0D^{-1}$, $F=DF_0D^{-1}$.
\end{proof} 
\begin{co} \label{mainc}
For any characteristic there exist matrices $G_0,Q_0,P_0,F_0$
such that Conditions (1) and (2) in Theorem \ref{main} are satisfied,
thus $m({\mathcal R})=3=m_p({\mathcal R})$ for any $p$.
\end{co}
\begin{proof}
There exist plenty of 2 by 2 invertible matrices 
$X$ such that if $m\in\Z\setminus
\{0\}$ then no entry of $X^m$ is a zero or a one, for instance if
$T=\sma{cc}t^2&t-1\\t+1&1\fma$ for $t$ any transcendental over the prime
subfield $\PP$ then it is easy to show by induction on $m\geq 2$
that the entries of $X^m$ are monic polynomials of degree $2m,2m-1,2m-1,2m-2$,
so $X=T^2$ will work here for positive $m$, and negative $m$ too as
$\mbox{det}(T)=1$. Let us take this to equal both
$X=\sma{cc}a&b\\c&d\fma$ and $Z=\sma{cc}\alpha&\beta\\\gamma&\delta\fma$,
whereas we will set
$Y=\sma{ll}\lambda&0\\0&\lambda^{-1}\fma$ for $\lambda$ any number which is
transcendental over $\PP(a,b,c,d)$, thus Condition (1) holds.

Let $X^m=\sma{cc} a_m&b_m\\c_m&d_m\fma$,
so that for $m\neq 0$
these entries are elements of $\PP(a,b,c,d)$ not equal to 0 or 1, 
and now consider a word $w$ in $F_2$ of the form
\[w=X^{m_1}Y^{n_1}X^{m_2}Y^{n_2}\ldots X^{m_k}Y^{n_k}\]
with all powers non zero. Then it is straightforward to show by induction
on $k\geq 2$ that the top left hand entry of $w$ is a 
Laurent polynomial in $\lambda$
of the following form:
\[\sum_{\bm{\epsilon}=(\epsilon_1,\ldots ,\epsilon_{k-1})\in\{-1,+1\}^{k-1}}
e^{(\bm{\epsilon})}_{m_1}e^{(\bm{\epsilon})}_{m_2}\ldots 
e^{(\bm{\epsilon})}_{m_k}
\lambda^{\epsilon_1n_1+\ldots +\epsilon_{k-1}n_{k-1}+n_k}\]
where each $e^{(\bm{\epsilon})}_{m_i}$ is equal to 
one of the four entries $a_{m_i},b_{m_i},c_{m_i},d_{m_i}$ of $X^{m_i}$.
This also holds for the other three entries of $w$,
except that now each $e^{(\bm{\epsilon})}_m$ will in general be a
different entry of $X^{m_i}$
and $n_k$ must be replaced by $-n_k$ for the two entries in the second
column.

This means we can pick out the leading and trailing terms of each entry
in $w$, which will be
\[
f_{m_1}f_{m_2}\ldots f_{m_k}
\lambda^{|n_1|+\ldots +|n_{k-1}|\pm n_k}+\ldots
+g_{m_1}g_{m_2}\ldots g_{m_k}
\lambda^{-(|n_1|+\ldots +|n_{k-1}|)\pm n_k}
\]
where the four different
$f_{m_i}$s (respectively $g_{m_i}$s) corresponding to a particular entry of $w$
are equal to 
the four different $e^{(\bm{\epsilon})}_{m_i}$s
when $\bm{\epsilon}$ is chosen to match (respectively oppose) the
signs of $(n_1,\ldots,n_{k-1})$, and where we take $+n_k$ for the first
column and $-n_k$ for the second. 

Now we have chosen $a_m,b_m,c_m,d_m$ and hence the four $f_m$
and $g_m$ so that 
they are never zero, thus we see that all entries of $w$ are Laurent
polynomials in $\lambda$ with more than one term and where each coefficient
is a product of non zero elements of $\PP(a,b,c,d)$. Thus $\lambda$ being
transcendental over this field means that no entry of $w$ lies in this
field and so cannot equal 0 or 1. Moreover every non identity
element of $F_2$ is either equal to $Y^n$ or will be in the form
above for some $w$ once we premultiply and/or postmultiply by a suitable
power(s) of $Y$. However
this just shifts all powers of $\lambda$ up or down by the same amount
in each entry. Consequently the same conclusion about the entries not lying
in this field still holds, unless under this pre/postmultiplication we end
up with $Y^n$ or $X^{m_1}Y^{n_1}$ but no matrix in the latter
case can contain a zero or 1 because of the condition we placed on the
powers of $X$. In particular $X$ and $Y$ generate a copy of
$F_2$ where no element has an entry equal to 1 or 0, except 
powers of $Y$. Thus Condition (2) is also satisfied.
\end{proof} 

We give a concrete example of such a faithful representation, at
least for characteristic zero. Here we can simplify matters by taking
$X=Z=\sma{cc}2&1\\1&1\fma^2$, because we can obtain the no zeros or
ones condition by induction using Fibonacci numbers. If we now
have any two transcendentally independent elements $\lambda,\phi$ of $\C$ 
then Corollary \ref{mainc} gives us the matrices $G,Q,P,F$, which can
now be returned to our original form on conjugation by $F_0^{-1}$. 
This allows us to take $G,Q,P=G_0,Q_0,P_0$ respectively, thus we have  
\[G=\sma{ccc}5&3&0\\3&2&0\\0&0&1\fma,
Q=\sma{ccc}\lambda&0&0\\0&\lambda&0\\0&0&\lambda^{-1}\fma,
P=\sma{ccc}\lambda&0&0\\0&\lambda^{-1}&0\\0&0&\lambda\fma
\]
but $F$ has become $F_0^{-1}DF_0D^{-1}F_0$ so we have
\[
F=\sma{ccc}-45 \phi^2+32 +18\phi^{-2}&0&-27\phi^2+18+12\phi^{-2}\\
0&1&0\\75\phi^2-45-27\phi^{-2}&0&45\phi^2-25-18\phi^{-2}\fma.\]

On looking through the (longer) list of graphs $\Gamma$
with 5 vertices, we can
use similar arguments to determine $m(R(\Gamma))$  or equivalently 
$m_p(R(\Gamma))$ for nearly all $\Gamma$. The ones that do not
succumb immediately using the above results are a few with a leaf edge
added to a graph of four vertices, which may well yield to a
similar proof as Theorem \ref{main} but one would prefer a general
approach rather than ad hoc arguments for each case, and a few with
an induced square whereupon we know we have $m(R(\Gamma))\geq 4$ by
Theorem \ref{rgsq} but we do not know if equality holds. 
Finally there is the $5$-cycle $\Gamma=C_5$
which has already been covered, at
least for characteristic zero where Theorem 1 of \cite{wangagtop07}
utilises the study of geodesics in symmetric spaces to
give a (discrete and) faithful representation of $R(C_5)$
into $GL(3,\C)$,
indeed into $SL(3,\F)$ for $\F=\Q(\sqrt{2},\sqrt{3},\sqrt{5})$,
thus this and Proposition \ref{2d}
tell us that $m(R(C_5))=m_0(R(C_5))=3$. We can now use the results
of \cite{kmkb} to obtain some immediate corollaries.
\begin{co} If $\Gamma$ is a finite forest, namely a graph with finitely
many connected components, each of which is a finite tree,
then $m_p(R(\Gamma))\leq 3$ for every characteristic $p$.
\end{co} 
\begin{proof}
Theorem 1.8 of \cite{kmkb} shows that the right angled Artin group
$R(\Gamma)$ is a subgroup of $\cal R$.
\end{proof}

We note that \cite{drms} Theorem 2 shows that all of these RAAGs are
3-manifold groups, and indeed these are almost exactly
the graphs $\Gamma$ such that $R(\Gamma)$ is a 3-manifold group; in general
there can also be components which are triangles, whereupon we will still
have $m_p(R(\Gamma))=3$. In particular any 3-manifold group which is also
a RAAG will have a faithful 3-dimensional linear representation but this
is not true for 3-manifolds in general, as will be mentioned in Section 4.

\begin{co} For the $n$-cycle $C_n$ we have 
$m_p(R(C_3))=1$ for every characteristic $p$, 
$m_p(R(C_4))=4$ for every characteristic
$p$ and $m_0(R(C_n))=3$ for $n\geq 5$ (although we do not know here if
$m_p(R(C_n))$ is finite for any positive $p$).
\end{co}
\begin{proof}
The group $R(C_3)$ is $\Z^3$, the group $R(C_4)$ is covered by
Corollary \ref{nosq} and \cite{kmkb} Theorem 1.12 shows that
the group $R(C_5)$ contains $R(C_m)$ for every $m\geq 6$.
\end{proof}
We finish this section with two
questions suggested by these results:\\
\hfill\\
(1) Is there a universal $N$ such that all RAAGs embed in $GL(N,\F)$?
In particular, do we have an example of a graph
$\Gamma$ such that $m(R(\Gamma))>4$?\\
\hfill\\
(2) Are all RAAGs linear in some positive characteristic, namely given any
RAAG $R(\Gamma)$ is there a prime $p$ 
with $m_p(R(\Gamma))<\infty$?\\
\hfill\\
If yes then (2) would have an interesting and powerful consequence, in
that a group which is not linear in any
positive characteristic cannot be virtually special (namely will have no finite
index subgroup that embeds in a RAAG). Note that for
$\Gamma$ a graph of $n$ vertices, 
the standard faithful linear representations of $R(\Gamma)$ into $GL(2n,\Z)$
obtained using Coxeter groups do not work in positive characteristic as
they contain many unipotent elements.  

\section{Braid groups and related groups}

\subsection{Braid groups}
The braid group $B_n$ is well known to have the following
presentation:
\[\langle \sigma_1,\ldots ,\sigma_{n-1}|
\sigma_i\sigma_j=\sigma_j\sigma_i\mbox{ if }|i-j|\geq 2,
\sigma_i\sigma_{i+1}\sigma_i=\sigma_{i+1}\sigma_i\sigma_{i+1}
\mbox{ for } 1\leq i\leq n-2\rangle.\]
Note that this implies all generators are conjugate to each other.
There is a considerable history regarding the question of 
whether braid groups are linear.
The braid group $B_3$ with presentation $\langle x,y|x^2=y^3\rangle$
has long known to be linear as it embeds in $\glw$ 
(as was first shown in \cite{magpkn} Lemma 2.1)
and 2 is clearly minimal so that $m(B_3)=m_0(B_3)=2$.
Indeed on taking 
\[x=\sma{ll} 0&s^3\\s^3&0\fma\qquad\mbox{ and }\qquad
y=\sma{rr} -s^2&-s^2\\-s^2&0\fma,\]
where $s$ is any transcendental over the prime subfield $\PP=\Q$ or
$\F_p$, it immediately follows from Proposition \ref{shl} that we
obtain a faithful representation of $B_3$ in $GL(2,\PP(s,t))$,
thus $m_p(B_3)=2$ for any characteristic $p$.

However the linearity
of $B_n$ for $n\geq 4$ remained open for many years. The well known
(reduced) Burau representation of $B_n$ in characteristic zero
(over the field $\Q(t)$ and indeed the ring $\Z[t^{\pm 1}]$
for $t$ an indeterminate) is $n-1$ dimensional
but it was shown by Moody in \cite{moopams} that this is unfaithful
for $n\geq 10$, with \cite{lop} taking this down to $n\geq 6$ and
\cite{bige} doing $n=5$. 
But the Lawrence-Krammer representation of dimension $n(n-1)/2$
(also in characteristic zero as it is over the field $\Q(t,q)$ or
even the ring $\Z[t^{\pm 1},q^{\pm 1}]$ for $t$ and $q$ independent
indeterminates)
which was found much later was shown to be faithful for all
braid groups by Bigelow in \cite{bigel} and by Krammer in \cite{kra},
so that $m_0(B_n)\leq n(n-1)/2$.

However we do not know
the exact value of $m_p(B_n)$ for any characteristic $p$ and any $n\geq 4$.
Indeed for $p>0$ we do not even know that $m_p(B_n)$ is finite. 
We now examine results in the literature on finite dimensional
representations of the braid groups $B_n$.

In \cite{dfg} it was shown in Proposition 2 that if 
$B_n$ has a faithful representation in $GL(\C,d)$ then it also has
a faithful irreducible representation in $GL(\C,d')$ for $d'\leq d$.
Then in \cite{for} the irreducible complex representations of $B_n$
having degree at most $n-1$ were obtained. We describe these results, 
first noting that a specialisation is the replacement of the indeterminate
$t$ with a non zero complex number, and that a tensor product of a 
given representation $\rho$ over $\C$
of $B_n$ with a 1-dimensional representation $\chi_y$ for $y\in\C^*$
can be thought of as replacing each matrix $\rho(\sigma_i)$ with
the scalar multiple $y\rho(\sigma_i)$ for $1\leq i\leq n-1$. 
Lemma 6 and Corollary 8 of this paper
show that most specialisations of
the Burau representation of $B_n$ are irreducible, although for each $n$
a finite number of specialisations of this $n-1$ degree representation split 
into an irreducible representation of degree $n-2$ and the trivial 
representation of degree 1.
Moreover Theorem 22 there states that if $n\geq 7$ then 
these are the only irreducible
representations over $\C$
of $B_n$ having degree $d$ for $1<d<n$, up to conjugacy and taking the tensor
product with a 1-dimensional representation. We take a brief look at the
question of faithfulness of these representations, showing that any one of them
being faithful implies that the Burau representation is faithful.

First of all any specialisation of the Burau representation being
faithful implies that the Burau representation itself is faithful too,
as the ring isomorphism 
$\Z[t^{\pm 1}]\rightarrow\Z[s^{\pm 1}]\subseteq\C$ extends to 
to a group homomorphism from
$\beta(B_4)$ to $\beta_s(B_4)$, where $\beta$ is the original 
Burau representation of $B_4$ over $\Z[t^{\pm 1}]$
and $\beta_s$ is the specialisation of $\beta$ at $s\in\C^*$.

Now suppose we have $s,y\in\C^*$ such that the tensor product
of $\chi_y$ with $\beta_s$ is a faithful representation. If the original
Burau representation $\beta$ of $B_n$ is not faithful then nor is
$\beta_s$, so any element in the kernel of $\beta_s$ would have to be
a scalar matrix in our faithful tensor product representation, and 
thus will lie in the centre of
$B_n$. This is well known to be an infinite cyclic subgroup
$\langle z\rangle$ for $z=(\sigma_1\sigma_2\ldots \sigma_{n-1})^n$ and
represented by $(-s)^nI_{n-1}$ under $\beta_s$, so $-s$ will be an $N$th root
of unity for some $N\geq 1$. But in the Burau representation $\beta$
each $\sigma_i$ is diagonalisable with eigenvalues 1 (repeated $n-2$ times)
and $-t$, so that we now have $\beta_s(\sigma_i^N)=I_{n-1}$ and thus
$\sigma_i^N$ maps to a scalar matrix in the tensor product representation.
As this representation is assumed faithful, we conclude that
$\sigma_i^N$ is in the centre of $B_n$ and thus is a 
scalar matrix in the original Burau representation $\beta$
by Schur's lemma,
as $\beta$ is irreducible. But actually
$\beta(\sigma_i^N)$ cannot be scalar by considering eigenvalues,
so we have a contradiction.

Given that the Burau representation is unfaithful 
as mentioned above, we can conclude that for $n\geq 7$ no irreducible
representation of $B_n$ over $\C$
with dimension $d\leq n-1$ is faithful and
thus the same is true for any $d$ dimensional representation. 
Thus we have $m_0(B_n)\geq n$ for $n\geq 7$, and in particular it tends to 
infinity with $n$.

The cases $n=4,5,6$ are also considered, whereupon there are a few
other irreducible representations in these small
dimensions. In particular for $n=4$ we already know that
$m_p(B_4)>2$ for all $p$ by the commutative
transitive comment in Section 2, so we have $3\leq m_0(B_4)\leq 6$ and
the same for $m(B_4)$. Again from this paper,
Theorem 13 shows that any irreducible representation
in degree 3 over $\C$ of $B_4$ is either the same as
those described above for $n\geq 7$
or one of two types of representation which clearly can never be
faithful. Thus the above argument also applies here to tell us that
there exists a faithful representation of the braid group $B_4$ into
$\glc$ if and only if the 3-dimensional
Burau representation is faithful.

As for positive characteristic, we can regard the Burau representation
as being over the field $\F_p(t)$ because we can reduce using the ring
homomorphism from $\Z[t^{\pm 1}]$ to $\F_p[t^{\pm 1}]$. A look at the
proofs in \cite{for} reveals that this result for $n=4$, as well as the above
statement that $m_0(B_n)\geq n$ for $n\geq 7$, goes through if we
replace $\C$ with our algebraically closed
universal coordinate domain in any positive
characteristic. We also need that a faithful representation implies
a faithful irreducible representation of no bigger degree, but again
the same argument goes through in positive characteristic. 
Thus we also have $m_p(B_n)\geq n$ for $n\geq 7$ and every $p$. 
The interesting point here about $n=4$ is that
Cooper and Long have shown in \cite{cl}
that the Burau representation is not faithful over $\F_2[t^{\pm 1}]$,
and also in \cite{clep} for $\F_3[t^{\pm 1}]$. 
It is hard to say
what this suggests for the characteristic zero case, though it follows
that showing $m_p(B_4)=3$ for any prime $p\geq 5$ would imply that the
characteristic zero Burau representation is faithful, which is
Question 1 in \cite{clep}. We note that Lemma 3 in \cite{leesong}
showed that for $B_3$ the associated Burau representation is faithful
in every positive characteristic.

We summarise the results
above in the following proposition:
\begin{prop} For the braid group $B_n$ we have $m_p(B_n)\geq n$ for $n\geq 7$ 
and every characteristic $p$, as well as $m_0(B_n)\leq n(n-1)/2$ though
we do not know if $m_p(B_n)$ is finite for $p>0$.

For $n=4$ we have $3\leq m(B_4)\leq m_0(B_4)\leq 6$ and 
$m(B_4)=3$ if and only if the Burau representation
is faithful. We do however know that
$m_2(B_4),m_3(B_4)>3$ but again we do not 
know that $m_p(B_4)$ is finite for any prime $p$. 
\end{prop}

\subsection{Aut and Out of a free group}

For the groups $Aut(F_n)$ and $Out(F_n)$ we do have results but they are
nearly all negative. It was shown in \cite{fp} using HNN extensions and the
representation theory of algebraic groups that $Aut(F_n)$ is not linear over
any field for $n\geq 3$, thus $m(Aut(F_n))=\infty$. It is clear that for
$n<m$ we have $Aut(F_n)\leq Aut(F_m)$ by fixing the last $m-n$ elements
of the free basis and back in \cite{mgtr} it was pointed 
out that $Aut(F_n)$ embeds in $Out(F_m)$ because $Aut(F_n)\cap Inn(F_m)$ 
is trivial.
(However we think that the unsupported claim made there that $Out(F_n)\leq
Out(F_m)$ must be a misprint.) It was also shown that if $Aut(n)$ has a
faithful representation then it has a faithful irreducible representation
of no higher degree. Although this point is now moot, it was used for
the proof in \cite{dfg} of the equivalent result for the braid groups
mentioned above.

However $Out(F_3)$ seems to behave somewhat differently, 
with its linearity over any characteristic currently open and to be found
as Problem 15.103 in the Kourovka Notebook \cite{kou}.
As for $Aut(F_2)$ and $Out(F_2)$, the latter is isomorphic to $GL(2,\Z)$
so we certainly have $m_0(Out(F_2))=m(Out(F_2))=2$. 
Finally $Aut(F_2)$ is linear, owing
to its close connection with the braid group $B_4$ as we now indicate.
For $F_2$ free on the elements $x,y$, it is well known that any element
$\alpha\in Aut(F_2)$ sends the commutator $c=[x,y]=xyx^{-1}y^{-1}$ to a
conjugate of $c^{\pm 1}$. Let $Aut^+(F_2)$ be the index 2 subgroup of
$Aut(F_2)$ such that $c$ is sent to a conjugate of itself (we call these
the orientation preserving automorphisms and the others orientation reversing,
which is a description that is also well defined in $Out(F_2)$).
Then the main theorem of \cite{dfg} from 1982
shows that the quotient of $B_4$ by its
centre $Z(B_4)$ is isomorphic to $Aut^+(F_2)$. 
We have already mentioned that this paper also shows that
a faithful complex representation of $B_n$ implies the existence of a
faithful irreducible complex representation of no higher degree.
Hence if $m_0(B_4)=d$ then this
representation is irreducible without loss of generality, thus
Schur's Lemma implies that the centre $Z(B_4)$
maps to scalar matrices and so
we can tensor with a 1-dimensional representation to obtain a 
faithful degree $d$ representation of $B_4/Z(B_4)\cong Aut^+(F_2)$, and
consequently a faithful degree $2d$ representation of $Aut(F_2)$.

Given that $d$ is known to be at most 6, we have that 
$Aut^+(F_2)$ is a subgroup of $GL(6,\C)$ and therefore $Aut(F_2)$ is a subgroup
of $GL(12,\C)$. 
Regarding lower bounds for $m_0(Aut(F_2))$, it 
was shown
in \cite{tenek} from 1986 that $Aut(F_2)$ cannot be embedded in $GL(5,\C)$.
As for $d=6$, of course we have by the above discussion that if the
Burau representation of $B_4$ is faithful then $Aut(F_2)$ is a subgroup
of $GL(6,\C)$. Moreover this paper also showed the converse, meaning that
we still only know $6\leq m_0(Aut(F_2))\leq 12$ and 
$3\leq m_0(Aut^+(F_2))\leq 6$.

\subsection{Mapping class groups}
Let ${\cal M}_g$ be the mapping class group of a closed orientable
surface of genus $g\geq 2$. It was shown in \cite{kork} and in
\cite{bgbd} using the braid
linearity results that ${\cal M}_2$ is linear, indeed the latter paper
obtains $m_0({\cal M}_g)\leq 64$. 
It is a famous unsolved question as
to whether ${\cal M}_g$ is linear for $g\geq 3$ 
so it seems that there is nothing more
to be said. We do however take note of \cite{flm} Theorem 1.6 which
implies that there are no low dimensional faithful representations. It
states that if $H$ is any finite index subgroup of ${\cal M}_g$ then
$m_0(H)\geq 2\sqrt{g-1}$.
 
\subsection{3-manifold groups}

Here a 3-manifold group will mean the fundamental group of a {\it compact}
3-manifold, so the group will be finitely presented. Linearity of such
groups has been studied over the years but a surprising consequence
obtained from applications of the recent Agol-Wise results is that most
3-manifold groups are linear even over $\Z$. Indeed on taking compact
orientable irreducible 3-manifolds $M^3$, we have that if $M^3$ admits
a metric of non-positive curvature then $\pi_1(M^3)$ is linear over
$\Z$. Linearity of some other special cases such as Seifert fibred spaces
is also known, meaning that amongst the fundamental groups of these
3-manifolds, linearity is only open for closed graph manifolds which
do not admit a metric of non-positive curvature. If these were also
linear over $\C$ then, as we can move up
or down subgroups of finite index (ie finite covers)
and take free products (ie connected sum) without affecting linearity,
every 3-manifold group would be linear. It is even possible that every 
3-manifold group could be linear over $\Z$.

In \cite{me3} we gave an example of one of these closed graph manifolds
where linearity of its fundamental group is unknown and showed that this group
did not embed in $GL(4,\F)$ for any field $\F$, thus answering a question
of Thurston. We note that the resulting 3-manifold was already known to be
virtually fibred, so that it is even unknown whether all semidirect
products of the form $\pi_1(S_g)\rtimes\Z$ (where $S_g$ is the closed
orientable surface of genus $g$) are linear. We also note that there seems to
be very little on linearity of 3-manifolds over fields of positive
characteristic.

However in the zero characteristic case, it should be said that these 
faithful linear representations over $\Z$ (obtained using virtually special
groups) are likely to be of vast size and so will be nowhere near the
minimum dimension of a faithful representation over $\C$. For instance all
finite volume hyperbolic orientable 3-manifold groups embed in $\slw$
but not in $SL(2,\Z)$. Indeed
the only example in the literature we could find
of a faithful linear representation over $\Z$ of any one of these 3-manifold
groups is in \cite{lr}, where Question 6.1 asks for any faithful representation
of such a group in $SL(3,\Z)$ and
Remark 7.1 points out that the figure 8 knot
group embeds in $SL(4,\Z)$.

Finally we note that amongst torsion free finitely presented groups,
the hyperbolic 3-manifold groups are precisely the ones admitting
faithful {\it discrete} embeddings in $\slw$ but there are plenty
of other 3-manifold groups with faithful
but non discrete embeddings.

\subsection{Word hyperbolic and random groups}

We first note that there do exist word hyperbolic groups which are
not linear over any field. However in the context of the usual
models of random groups, where most groups turn out to be word
hyperbolic, a big breakthrough of the Agol-Wise results was that
at low densities, and certainly for models involving a fixed
number of generators and relators, a random group is virtually
special and so linear over $\Z$. 

As for the minimum dimension,
in a word hyperbolic group all centralisers are virtually cyclic so
our techniques for showing that there are no low dimensional faithful
representations will not work here. Indeed closed hyperbolic orientable
3-manifold groups and hyperbolic limit groups, for instance, will embed
in $\slw$. However we do have one result on the lack of 2-dimensional 
embeddings which was outlined in \cite{flr} and \cite{ds}.
\begin{prop}
If $\langle x,y\,|\,r(x,y)\rangle$ is a random 2-generator 1-relator
presentation defining the group $G_r$ with cyclically reduced word
$r$ having length $l$ then the probability that $m(G_r)>2$ tends to
1 exponentially as $l$ tends to infinity.
\end{prop}
\begin{proof}
Theorem 3.1 of \cite{flr} (based on results of Magnus) shows that if
$G_r$ has the above presentation and is not metabelian but admits a
faithful representation into $\slw$ then $r(x,y)$ must be conjugate
to $r(x^{-1},y^{-1})$ or $r^{-1}(x^{-1},y^{-1})$ in the free group
on $x,y$. This can also be seen in any field $\F$ by using the
``hyperbolic involution'': given any $X,Y\in\sltf$ we have that $ZX=X^{-1}Z$
and $ZY=Y^{-1}Z$ for $Z$ the 2 by 2 matrix over $\F$ defined by $Z=XY-YX$.
Thus if $\mbox{det}(Z)\neq 0$ then $Z\in\glw$ conjugates $X$ to $X^{-1}$
and $Y$ to $Y^{-1}$, in which case we have the above result by applying
the same theorem of Magnus on 1-relator presentations. If $Z$ has zero
determinant then $\langle X,Y\rangle$ will be metabelian, as can be seen
because now $XYX^{-1}Y^{-1}-I_2$ has zero determinant, so $XYX^{-1}Y^{-1}$
has 1 as a repeated eigenvalue. This means we can apply \cite{me3} Proposition
A.1 to conclude that $\langle X,Y\rangle$ is metabelian. This Proposition
is stated for $\C$ but we only need part (i) which works for an algebraically
closed field of any characteristic. But without loss of generality we
can assume that $\F$ is algebraically closed, which we do for the
rest of the proof.

Now suppose that $G_r$ embeds in $\gltf$, with $x,y$ sent to matrices
$X,Y$. On replacing these with $\lambda X$ and $\mu Y$ where $\lambda,\mu\in
\F$ are such that $\lambda^2=1/\mbox{det}(X)$ and $\mu^2=1/\mbox{det}(Y)$,
we obtain a new embedding of $G_r$ in $\sltf$ provided there were
no scalar matrices except 
$I_2$ in the original embedding.
But generically $G_r$ will be non-elementary word hyperbolic,
so will not be metabelian and have no centre. 
Hence we now know $r(x,y)$ is conjugate 
in $F_2$ to $r^{\pm 1}(x^{-1},y^{-1})$ which is also cyclically reduced
and of length $l$. This means that either $r(x^{-1},y^{-1})$ or
$r^{-1}(x^{-1},y^{-1})$ is a cyclic permutation of $r(x,y)$. In the
first case there will be $k\leq l/2$ where we can take either the first $k$
or the last $k$ letters of the word $r(x,y)$ and move them to the back or
to the front to obtain $r(x^{-1},y^{-1})$, whereupon we see that these
$k$ letters determine the whole word, so that $r$ lies in an exponentially
small subset of the cyclically reduced words of length $l$. A similar
argument works in the second case, whereupon we see at most $l/2+1$ letters
determine the whole word.
\end{proof}

\section[Representations of free by cyclic groups]{Low 
dimensional representations of free by cyclic groups}
In this section, a free by cyclic group will mean the semidirect product
$F_n\rtimes_\alpha\Z$ formed by taking an automorphism $\alpha$ of the
free group $F_n$ of rank $n\geq 2$. (Thus ``free'' here means non abelian
free of finite rank and ``cyclic'' means infinite cyclic.) Given
$\alpha\in Aut(F_n)$ (or rather in $Out(F_n)$ as automorphisms that are
equal in $Out(F_n)$ give rise to isomorphic groups), 
we are interested in the possible values of 
$m(F_n\rtimes_\alpha\Z)$. However we will see in this section 
that sometimes it is more straightforward to obtain results on
characteristic zero representations, whereupon we are
really looking at $m_0(F_n\rtimes_\alpha\Z)$.

The following result is
folklore.
\begin{lem} \label{inout}
If $\Gamma$ is a group with no centre and 
$\alpha$ is an automorphism of $\Gamma$ that has infinite order
in $Out(\Gamma)$
then the semidirect product $G=\Gamma\rtimes_\alpha\Z$ embeds
in $Aut(\Gamma)$.
\end{lem}
\begin{proof} 
The no centre condition means that $Aut(\Gamma)$ contains a copy of
$\Gamma$ in the form of the inner automorphisms $Inn(\Gamma)$,
whereupon we write
$\iota_\gamma$ for the element of $Inn(\Gamma)$ which is conjugation
by $\gamma\in\Gamma$.
Now for any
$\alpha\in Aut(\Gamma)$ and $\gamma\in\Gamma$ we have $\alpha\iota_\gamma
\alpha^{-1}=\iota_{\alpha(\gamma)}$. Thus on taking the subgroup
$\langle \alpha,Inn(\Gamma)\rangle$ of $Aut(\Gamma)$ along with the map from
$G$ to this subgroup that sends 
$\gamma$ to $\iota_\gamma$ and the generator $t$ of $\Z$ to $\alpha$, we
have that all relations in $G$ are preserved so our map extends to a
surjective homomorphism. Now any element of $G$ can be written in the
form $\gamma t^n$. If this maps to the identity then $\alpha^n$ would
equal the inner automorphism $\iota_{\gamma^{-1}}$.
\end{proof}

\begin{prop} If the group $\Gamma$ has no centre
and  $Aut(\Gamma)$ embeds in $GL(d,\F)$ for $\F$ some field of characteristic
$p$ then any group of the form $G=\Gamma\rtimes_\alpha\Z$ has $m_p(G)\leq d$.
\end{prop}
\begin{proof}
By Lemma \ref{inout}, either $G$ is a subgroup
of $Aut(\Gamma)$ or $\alpha$ has order $k$ in $Out(\Gamma)$. In the latter
case we take the subgroup $\langle \alpha,Inn(\Gamma)\rangle$ of
$GL(d,\F)$ which provides a homomorphic image of $\Gamma\rtimes_\alpha\Z$. 
This homomorphism will not be injective because $\alpha^k\in Inn(\Gamma)$ but
this can be bypassed by replacing $\alpha$ with the element $\lambda\alpha$,
where $\lambda$ is a scalar which is transcendental over $\F$, and then
replacing $\F$
with $\F(\lambda)$. If this new homomorphic image were not injective then
there is $k>0$ with $\lambda^k\alpha^k\in Inn(\Gamma)\leq GL(d,\F)$
which is a contradiction as $\mbox{det} (\lambda^k\alpha^k)=
\lambda^{dk}\mbox{det}^k(\alpha)$ is not in $\F$.
\end{proof}

On now taking $\Gamma=F_n$, this result does not actually help when
$n\geq 3$ and indeed for free by cyclic groups where the free group
has rank at least 3, the question of linearity is open and will be
discussed in the next sections. But returning to $n=2$,
as $m_0(Aut(F_2))\leq 12$, the same bound holds for
$m_0(F_2\rtimes_\alpha\Z)$ over all automorphisms $\alpha$. Indeed the
argument above also tells us that $m_0(F_2\rtimes_\alpha\Z)\leq 6$
if $\alpha$ is orientation preserving, because $m_0(Aut^+(F_2))\leq 6$.
In this section we will give some better bounds, and in particular 
will find the exact value of $m(F_2\rtimes_\alpha\Z)$
and $m_0(F_2\rtimes_\alpha\Z)$ when $\alpha$ is orientation preserving
for all but one family of automorphisms. 

On taking a matrix 
$(I\neq)M\in SL(2,\Z)$, which we regard as the group of orientation
preserving outer automorphisms of $F_2$ denoted by $Out^+(F_2)$,
we have that $M$ is hyperbolic, elliptic or parabolic. In the first case
the automorphism $\alpha$ would be a pseudo-Anosov homeomorphism of the
once punctured torus and Thurston showed that the corresponding mapping torus
has a hyperbolic structure, so that $F_2\rtimes_\alpha\Z$ embeds in $PSL(2,\C)$
and can be lifted to $\slw$, thus giving 
$m_0(F_2\rtimes_\alpha\Z)=m(F_2\rtimes_\alpha\Z)=2$ here. However we do not
know for these groups if $m_p$ is finite when $p>0$.

If $M$ is elliptic then $\alpha$ will have finite order. 
In this case
$F_2\rtimes_\alpha\Z$ cannot embed in $\slw$ as it contains $F_2\times\Z$
which is not commutative transitive. However we will still find that
$m_0(F_2\rtimes_\alpha\Z)=m(F_2\rtimes_\alpha\Z)=2$. Indeed this holds
for finite order orientation reversing automorphisms too which we now
consider:
\begin{lem} \label{orrev}
If $\alpha$ is an orientation reversing outer automorphism considered as
a matrix $A=\sma{cc}a&b\\c&d\fma\in\glz$ with $ad-bc=-1$ then either
$A^2$ is hyperbolic or $A$ is of order 2.
\end{lem}
\begin{proof} We can determine the type of $A^2$ by the trace, which
is $(a+d)^2+2\geq 2$. Thus $A^2$ is hyperbolic unless $d=-a$, in which
case $A^2=I$. 
\end{proof}

It is a well known fact that the only order 2 element in $\slz$ is
$-I$ and in $\glz$ all order 2 elements with determinant $-1$ are
conjugate either to
$\sma{rr}1&0\\0&-1\fma$ or $\sma{cc}0&1\\1&0\fma$.
This leads to a complete answer for finite order automorphisms.
\begin{prop} \label{finite}
If $[\alpha]\in Out(F_2)$ has finite
order then the group $F_2\rtimes_\alpha\Z$ embeds in $\glw$
and hence $m(F_2\rtimes_\alpha\Z)=m_0(F_2\rtimes_\alpha\Z)=2$.
\end{prop}
\begin{proof}
If our statement holds for $\alpha$ then it holds for any power of
$\alpha$ too. Let us first take $\alpha$ to be orientation preserving.
There we have that $Out^+(F_2)=SL(2,\Z)$ is isomorphic to the amalgamated
free product $C_6*_{C_2}C_4$ and so any finite order element of $GL(2,\Z)$ is
conjugate to a power of the generator of $C_6$ or of $C_4$.

We use the well known fact that if we have two ordered pairs $(A,B)$ and
$(X,Y)$ of elements of $SL(2,\C)$ then we have $T\in SL(2,\C)$ with
$TAT^{-1}=X$ and $TBT^{-1}=Y$ if and only if the two trace triples 
$(tr(A),tr(B),tr(AB))$ and $(tr(X),tr(Y),tr(XY))$ are equal
points in $\C^3$, provided that $tr(ABA^{-1}B^{-1})\neq 2$ which would
mean that $\langle A,B\rangle$ is a soluble group anyway.

Let $F_2$ be free on $x,y$ and
$\alpha$ be the automorphism $\alpha(x)=y^{-1},\alpha(y)=xy$ which is
represented by the order 6
matrix $\sma{rr} 0&1\\-1&1\fma\in\slz$. Suppose
that we have a pair of matrices $(A,B)$ with trace triple $(a,b,c)$
which generates $F_2$. Then we see that
the trace triple of $(\alpha(A),\alpha(B))$ is $(b,c,a)$. Therefore $a=b=c$ is
necessary for $T\in\slw$ to exist, and as 
$tr(ABA^{-1}B^{-1})=a^2+b^2+c^2-abc-2$,
we have that $a\neq -1,2$ is sufficient. We can assume that $T$ has infinite
order by replacing $T$ with $\lambda T\in\glw$. Finally we need a value of
$a\in\C$ which ensures that $A,B$ really do generate a rank 2 free group.
But the trace triples of generating pairs in $SL(2,\R)$ which give
rise to discrete free groups of rank 2 are completely understood:
for instance $a,b,c>2$ and $a^2+b^2+c^2-abc<0$ is sufficient so any
$a>3$ will do. 

Similarly the generator of $C_4$ can be taken to be $\alpha(x)=y^{-1}$
and $\alpha(y)=x$. We then require our trace triple to satisfy
$(a,b,c)=(b,a,ab-c)$ which implies $b=a$ and $c=a^2/2$. Then like before,
$a>\sqrt{8}$ will provide a rank two free subgroup of $SL(2,\R)$
admitting the required automorphism under conjugation by some infinite
order element $T\in\glw$.

Finally we need to consider the orientation reversing case obtained from
the elements $\sma{rr}1&0\\0&-1\fma$ and $\sma{cc}0&1\\1&0\fma$. In the first
case $\alpha(x)=x$ and $\alpha(y)=y^{-1}$ so the trace triple equation is
$(a,b,c)=(a,b,ab-c)$, giving $c=ab/2$ and we can pick say $a=b=4$ and $c=8$
to show the existence of an embedding of the corresponding free by cyclic
group into $\glw$. We have $\alpha(x)=y$ 
and $\alpha(y)=x$ in the second 
case, whereupon we obtain the equation $(a,b,c)=(b,a,c)$ with solution
$(a,a,c)$ which will again provide a rank 2 free group if $a=4$ and $c=8$.
\end{proof}

Note: this proof ought to work in positive characteristic as well but
what seems to be lacking here is a suitably general sufficient condition
for a pair of 2 by 2 matrices to generate a rank 2 free group. However,
as $m_p(F_2\times\Z)=2$ for every characteristic $p$, we at least know that
if $\alpha$ has order $d$ in $Out(F_2)$ then $m_p(F_2\rtimes_\alpha\Z)\leq 2d$
for all $p$ by induced representations.
 
Moving on to orientation reversing elements $\alpha$
whose square is hyperbolic, we have mentioned that
$m_0(F_2\rtimes_{\alpha^2}\Z)=2$ and so $m_0(F_2\rtimes_\alpha\Z)\leq 4$
by induced actions. However it is also clear that 
$m(F_2\rtimes_\alpha\Z)\neq 2$ because there will be an element $W$
in the group $F_2=\langle A,B\rangle$ such that
$T$ conjugates $ABA^{-1}B^{-1}$
to $W^{-1}(ABA^{-1}B^{-1})^{-1}W$. Hence $(WT)^2$ commutes with $WT$
and $ABA^{-1}B^{-1}$ but they do not commute with each other, meaning that
$(WT)^2$ must be in the centre of $\langle A,B,T\rangle$
if there is a faithful 2 dimensional
representation. This would imply that conjugation by $(WT)^2$ would be the 
identity automorphism of $F_2$, but conjugation by $W$ is inner
and thus that $\alpha$ has order 2 in $Out(F_2)$. Consequently
$m(F_2\rtimes_\alpha\Z)$ and $m_0(F_2\rtimes_\alpha\Z)$
can only take the value 3 or 4. However these groups have an index 2 subgroup
which is the fundamental group of a finite volume hyperbolic 3-manifold
and in particular the centraliser of any element will be virtually $\Z$
or virtually $\Z^2$, so the techniques we have been using to determine the
minimum dimension of a faithful representation will not apply here. 

This leaves only the orientation preserving parabolic case where
$M=\pm \sma{cc}1&n\\0&1\fma$ for $n\in\N$ without loss of generality,
by replacing $M$ with $M^{-1}$. The question of linearity has received
some attention, indeed whether $F_2\rtimes_\alpha\Z$ is linear for
$M=\sma{cc}1&1\\0&1\fma$, which is $\alpha(x)=x$ and $\alpha(y)=yx$, was
Question 18.86 in the Kourovka Notebook \cite{kou} but \cite{barbry} and
\cite{rom} pointed out that linearity of the braid group $B_4$ provides
a positive answer for this group, as well as for all the groups in the
family as they are commensurable. Here we can provide a quick proof that
does not use the linearity of the braid group and which moreover
works in all characteristics.
\begin{prop} For the group $G=F_2\rtimes_\alpha\Z$ where
$\alpha(x)=x$ and $\alpha(y)=yx$, we have $m_p(G)\leq 6$ for all
characteristics $p$.
\end{prop}
\begin{proof} Let the stable letter of $G$ be $t$, so
that $G=\langle x,y,t\rangle=\langle y,t\rangle$ because $y^{-1}tyt^{-1}=x$.
It was shown in \cite{nbws} that the index 2 subgroup 
$H=\langle y^2,t,y^{-1}ty\rangle$ 
of $G$ embeds into the RAAG $\mathcal R$
in Theorem \ref{main},
indeed using our notation there for the generators of $\mathcal R$ we
can take $y^2=GF^{-1},t=Q$ and $y^{-1}ty=P$. As $H$ fails the
commutative transitive property and
we showed that
this RAAG embeds in $GL(3,\F)$ for every characteristic $p$, we conclude
that $m_p(H)=3$ for all $p$ and thus that $3\leq m_p(G)\leq 6$.
\end{proof} 

In conclusion we have a reasonable understanding of the minimum dimension
in characteristic zero of faithful linear representations for
groups of the form $F_2\rtimes_\alpha\Z$ 
but less so for positive characteristic.
Moreover we have not ruled out that $m_p(F_2\rtimes_\alpha\Z)$ is at most
3 for every $p$ and every $\alpha$; indeed showing that 
$m_0(F_2\rtimes_\alpha\Z)>3$ for some orientation preserving $\alpha$ would
imply that the Burau representation for $B_4$ is not faithful, which seems
to be the outstanding open question in this area. We will now see that things
are very different when the rank of the free group is greater than 2.

\section[Not linear in positive characteristic]{The Gersten free by 
cyclic group is not linear in positive characteristic}

We have already mentioned that the linearity question is open for free
by cyclic groups $F_n\rtimes_\alpha\Z$ where $n\geq 3$, although the
Agol-Wise results again have been used to show this holds in most cases.
Indeed \cite{hgws2} proved that any word hyperbolic free by cyclic group
acts properly and cocompactly on a CAT(0) cube complex, thus these groups
will be virtually special and so have faithful linear representations over
$\Z$. Therefore, as not being word hyperbolic is equivalent to containing
$\Z\times\Z$ for free by cyclic groups, our question should really be
stated as:
\begin{qu}
If $F_n\rtimes_\alpha\Z$ is a free by cyclic group for $n\geq 3$ which
contains $\Z\times\Z$ then is it linear over some field?
\end{qu} 

In Gersten's paper \cite{ger} the free by cyclic group 
$F_3\rtimes_\alpha\Z=\langle a,b,c,t\rangle$
with $tat^{-1}=a,tbt^{-1}=ba,tct^{-1}=ca^2$ is introduced and
shown to have very strange properties. In particular an argument
using translation length proves that it cannot act properly and
cocompactly by isometries on a CAT(0) space. Now given the open
question of whether all free by cyclic groups $F_n\rtimes_\alpha\Z$
for $n\geq 3$ are linear, Gersten's group $G$ would seem like an
important test case. In this section we 
prove that $G$ is not linear
over any field of positive characteristic. This will be a consequence
of showing that the most straightforward linear representations of this
group, namely ones where 
the elements $t,a$ are both diagonalisable and thus simultaneously
diagonalisable as $ta=at$, are never faithful over any field.

\begin{thm} \label{pos}
Suppose we have commuting elements $T,A\in GL(d,\F)$ for $d$ any dimension
and $\F$ any field, such that the matrix
$T$ is conjugate to $TA$ and also to $TA^2$. Then the eigenvalues
of $A$ must be roots of unity.
\end{thm}
\begin{proof}
We replace $\F$ by its algebraic closure, which we will also call
$\F$, and first
suppose that both $A$ and $T$ are diagonalisable, so that we can choose
a basis $e_1,\ldots,e_d$ in which both
\[T=\mbox{diag}(t_1,\ldots ,t_d)\mbox{ and }
A=\mbox{diag}(a_1,\ldots ,a_d)\]
are simultaneously diagonal. As $TA$ and $TA^2$ are also then diagonal,
each has entries which are a permutation of the diagonal entries of
$T$. Although these permutations, which we will name $\pi$ and
$\sigma$ respectively, will not in general be well defined because of
repeated eigenvalues, we choose appropriate $\pi$ and $\sigma$ defined
in some suitable way. We now permute our basis so that $\pi$ is a
product of disjoint consecutive cycles, that is
\[\pi=(1 2\ldots k_1)(k_1+1 \,k_1+2 \ldots k_2)\ldots (k_{r-1}+1 \ldots k_r)\]
for $k_r=d$.

First suppose that the number of cycles $r$ in $\pi$ is 1. Then we have
\[T=\mbox{diag}(t_1,\ldots ,t_d)\mbox{ and }
TA=\mbox{diag}(t_2,\ldots ,t_d,t_1)\]
so that
\begin{eqnarray*}A=\mbox{diag}
(t_2/t_1,t_3/t_2,\ldots ,t_1/t_d)\mbox{ and thus }
TA^2&=&\mbox{diag}(t_2^2/t_1,t_3^2/t_2,\ldots ,t_1^2/t_d)\\
&=&\mbox{diag}(t_{\sigma(1)},t_{\sigma(2)},\ldots ,t_{\sigma(d)})
\end{eqnarray*}
for the permutation $\sigma$ above.

Now all $t_i$ are in the abelian group $\F^*$ written multiplicatively,
but on changing to additive notation we can regard the two expressions
for $TA^2$ as providing a system of linear equations. Thus let us
work in an arbitrary abelian group $\A$ written additively, so
that we are replacing $(\F^*,\times)$ above with $(\A,+)$. We are
hence trying to solve the homogeneous system of equations
\[2x_2=x_1+x_{\sigma(1)},2x_3=x_2+x_{\sigma(2)},\ldots ,2x_1=x_d
+x_{\sigma(d)}\]
in the unknown variables $x_1,\ldots ,x_d\in \A$. We note that there
exist solutions where $x_1=\ldots =x_d$ (which would result in $A$
being the identity) and we are trying to rule out other solutions.

Let us first set $\A=\R$, so that we can use the usual order on $\R$ as
well as linear algebra. Given a non zero solution
$(r_1,\ldots ,r_d)\in\R^d$, let $M$ be $\mbox{max}_{1\leq i\leq d}|r_i|$
and let $|r_k|$ attain $M$, so that $r_k=M>0$ without loss of generality
by multiplying the solution by $-1$ if necessary. Now one equation is
$2r_k=r_{k-1}+r_{\sigma(k-1)}$ (where all subscripts are taken modulo
$d$) so $2M=r_{k-1}+r_{\sigma(k-1)}\leq |r_{k-1}|+|r_{\sigma(k-1)}|\leq 2M$,
thus for equality we need $r_{k-1}$ and $r_{\sigma(k-1)}$ both to have
modulus $M$ and be positive. Thus now we can replace the equation with
left hand side $2r_k$ by the equation having left hand side $2r_{k-1}$
and continue until we have the constant solution.

However this assumed that the initial permutation $\pi$ was just the
cycle $(1 2 \ldots d)$. Let us consider the general case
\[\pi=(1 2\ldots k_1)(k_1+1 \,k_1+2 \ldots k_2)\ldots (k_{r-1}+1 \ldots k_r)\]
so that the two expressions for $TA^2$ now read
\begin{eqnarray*}
&&\mbox{diag}(t_2^2/t_1,\ldots,t_1^2/t_{k_1},t_{k_1+2}^2/t_{k_1+1},
\ldots ,t_{k_1+1}^2/t_{k_2},\ldots ,t_{k_{r-1}+1}^2/t_d)\\
&=&\mbox{diag}(t_{\sigma(1)},\ldots ,t_{\sigma(k_1)},t_{\sigma(k_1+1)},
\ldots ,t_{\sigma(k_2)},\ldots,t_{\sigma(d)}).
\end{eqnarray*}
Then if $S$ is the subgroup of Sym($d$) generated by $\pi,\sigma$, we
have that the orbits of $S$ are unions of the disjoint cycles for
$\pi$ as above. However if $S$ is not transitive then it should be
clear that we have solutions $x_1,\ldots ,x_d\in\A$ which are
constant on orbits of $S$ but which are not constant overall, because
the equations in separate orbits involve disjoint sets of variables
(however these solutions still give rise to the matrix $A$ being
the identity). Therefore let us consider the orbit $O$ under $S$ of
some point $x\in\{1,\ldots ,d\}$ which will be a union of the cycles
for $\pi$. Let us assume without loss of generality that $j\in O$ is such
that $x_j=M>0$ maximises $|x_i|$ over all $i\in O$. Then $j$ sits
in some cycle $(k_{l-1}+1\, \ldots \,k_l)$ and as before on considering
the equation $2r_j=r_{j-1}+r_{\sigma(j-1)}$, then $2r_{j-1}=r_{j-2}+
r_{\sigma(j-2)}$ and so on, where our subscripts are taken from numbers
in this cycle and where by subtracting 1 we mean shifting backwards
round the cycle. This implies not only that
\[r_{k_{l-1}+1}=r_{k_{l-1}+2}=\ldots =r_{k_l}=M\]
but also that any subscript $s$ which is an image under $\sigma$ of 
a point in this cycle will satisfy $r_S=M$ too. Thus we now move to another 
cycle until we see that our solution is constant on the whole of $O$.

We now deduce the same conclusion for solutions of these equations
over any torsion free abelian group $\A$. If we have a solution
$(x_1,\ldots ,x_d)\in\A^d$ then we replace $\A$ with the finitely
generated subgroup $\langle x_1,\ldots ,x_d\rangle=\A_0$ and work
in $\A_0$ which, being a finitely generated torsion free abelian
group, is just a copy of $\Z^m$ for some $m$. If we now take a
particular $\Z$-basis for $\A_0$, we can express $x_1,\ldots ,x_d$ as
elements of $\Z^m$ and then our $d$ equations become $m$ lots of
$d$ equations, with one set of $d$ equations for each coefficient
of $\Z^m$. But as each system of equations over $\Z$ can be thought of as
also over $\Q$ and indeed over $\R$, our above argument tells us that
our solution must be constant over each orbit coordinate-wise, so
indeed our elements $x_1,\ldots ,x_d$ are equal amongst subscripts in
the same orbit.

However this assumed that $\A$ is torsion free, whereas in the
multiplicative group $\F^*$ of a field we will have roots of
unity. To deal with this case, note that if the element $X\in G$
is such that $XTX^{-1}=TA$ then $XT^nX^{-1}=T^nA^n$ for any $n\in\Z$
and similarly we have $YT^nY^{-1}=T^nA^{2n}$ if $YTY^{-1}=TA^2$.
Thus on initially being given our diagonal elements $t_1,\ldots ,t_d\in\F^*$
of $T$, where now we finally return to multiplicative notation, we have
that $\langle t_1,\ldots ,t_d\rangle$, considered as an abstract
finitely generated abelian group, must be isomorphic to $\Z^r\oplus R$
for some $r\leq d$ and $R$ a finite subgroup consisting of the torsion
elements. Hence there exists an exponent $e>0$ such that 
$t_1^e,\ldots ,t_d^e$ all lie in the $\Z^r$ part and so these elements
generate a torsion free abelian subgroup. Now we can can run through
the whole proof above with $T$ and $A$ replaced by $T^e$ and $A^e$
(but the conjugating elements $X$ and $Y$ remain the same), whereupon
we conclude that $A^e$ must be the identity and so the eigenvalues
of $A$ are all roots of unity.

If our elements are not diagonalisable then, as they commute, we can
still find a basis in which both
$T$ and $A$ are upper triangular. We can then
work through the above proof using the diagonal elements of $T$ and of $A$,
which will multiply in the same way to give the diagonal entries
of $TA$ and of $TA^2$, 
whereupon we will conclude that some power $A^e$ of $A$ is upper triangular
with all ones down the diagonal, thus again the eigenvalues of $A$ are all
roots of unity.
\end{proof}
\begin{co}
If a group $G$ has commuting elements $T,A$ with $A$
of infinite order
such that $T,TA,TA^2$ are all conjugate in $G$ then it is not
linear over any field $\F$ of positive characteristic.
\end{co}
\begin{proof}
We can assume $\F$ is algebraically closed, whereupon Theorem \ref{pos}
tells us that a power $A^e$ of $A$ is unipotent, which over
positive characteristic implies that $A$ has finite order.
\end{proof}
\begin{co} \label{nops}
Gersten's group $G$ is not linear in any positive characteristic.
\end{co}
\begin{proof}
We have $b^{-1}tb=at=ta$ and $c^{-1}tc=a^2t=ta^2$ with $a$ of infinite
order.
\end{proof}

Gersten constructed this example specifically
because it embeds in $Aut(F_3)$ as in Lemma \ref{inout}, and thus
in $Aut(F_m)$ for $m\geq 3$ and $Out(F_n)$ for $n\geq 4$. Consequently
we have a proof that none of these groups are linear over a field of
positive characteristic without using the theory of algebraic groups
in \cite{fp}. However as long as the question of linearity of Gersten's
group over $\C$ is still open, we will not know if this approach can
be made to work in characteristic zero.

\section[No low dimensional faithful representation]{No low dimensional 
faithful representation of Gersten's group}

Although we are not able to resolve linearity of Gersten's group over
$\C$, we show here that it has no low dimensional faithful linear
representation over $\C$ and hence over any field. This establishes for
free by cyclic groups the equivalent result in \cite{me3} for
3-manifold groups. Indeed our method of proof will be similar in that
we work through and eliminate the various possible Jordan normal forms, thus
we will quote straight from \cite{me3} when needed. However here the
actual arguments for eliminating the various Jordan normal forms will
generally be shorter and, unlike the earlier paper, no case will require
any involved calculations.    

\begin{thm} There is no faithful representation over $\C$ of Gersten's
group $G$ in dimension 4, thus by this and Corollary \ref{nops} we
have $m(G)>4$.
\end{thm}
\begin{proof}
From Theorem \ref{pos} we can assume that the eigenvalues of $A$ are all
roots of unity, or even all 1
by replacing $A$ with $A^n$ and $T$ with $T^n$ for the appropriate
integer $n$. Now $A$ commuting with
$T$ implies that $A$ sends each generalised eigenspace of $T$ to itself
(the same is true the other way round but now the generalised 1-eigenspace
for $A$ is all of $\C^4$). Consequently we can split our matrix into blocks,
in each of which $T$ has the same eigenvalue and then conjugate our
representation within each block such that $T$ is in Jordan normal form
or a canonical form very close to this. As the centraliser of $T$ 
in the Gersten group contains
the rank 2 free group $\langle A,ABA^{-1}B^{-1}\rangle$, we must have that
at least one block in the matrix representation for $T$ has a centraliser
big enough to contain non abelian free groups. We then work out the possible
forms for $T$ and then for $A$ by looking at the resulting centraliser
and we see when $T$ can be conjugate both to $TA$ and to $TA^2$. We first
list the Jordan blocks of size at most 4 by 4 with a big enough centraliser,
which is Proposition 5.3 in \cite{me3}.

\begin{prop} \label{cent} 
Given a Jordan block of size at most 4 by 4, the following
are the ones with a big centraliser 
(meaning they contain a non abelian free group):
\[\lambda I_2,\,\lambda I_3,\,\lambda I_4,\,
\sma{cccc}\lambda&1&0&0\\0&\lambda&0&0\\0&0&\lambda&0\\
0&0&0&\lambda\fma,
\sma{cccc}\lambda&1&0&0\\0&\lambda&0&0\\0&0&\lambda&1\\
0&0&0&\lambda\fma.
\]
\end{prop}

Starting with the smallest such block which is $\lambda I_2$, the other
block(s) appearing in $T$ need not have big centralisers so the
possibilities are $\mu I_1$ and $\nu I_1$, or $\mu I_2$ or
$\sma{cc}\mu&1\\0&\mu\fma$. Now we can conjugate $A$ within the
$\lambda I_2$ block without changing $T$, so we can assume that
this part of $A$ equals
$\sma{cc}1&1\\0&1\fma$ or $I_2$. But the first option means that
$TA$ does not
have $\lambda$ as an eigenvalue of geometric multiplicity 2, unlike $T$.
As for $I_2$, either $A$ would be diagonal and so equal to the
identity if $T$ has two 1 by 1 blocks, or we can apply the above to
$\mu$ if $T$ has a $\mu I_2$ block, or we have 
\[T=\sma{cccc}\lambda&0&0&0\\0&\lambda&0&0\\0&0&\mu&\epsilon\\0&0&0&\mu\fma
\qquad\mbox{and}\qquad
A=\sma{cccc}1&0&0&0\\0&1&0&0\\0&0&1&1\\0&0&0&1\fma\]
for $\epsilon=0$ or 1,
where $T$ is conjugate to $TA$ and $TA^2$ for $\mu\neq -1,-1/2$
but it is easily checked that such a conjugating matrix $B$ or $C$ would also
split into the same 2 by 2 diagonal blocks with the bottom one upper
triangular, thus $\langle A,B\rangle$ nor $\langle A,C\rangle$ could be
a rank 2 free group.

We next assume that $A$ and $T$ split into generalised eigenspaces
for $T$ which are a 3 by 3 block and a 1 by 1 block, thus these have
to equal $\lambda I_3$ and $\mu I_1$ because of the restriction of some 
block having a big centraliser. Consequently $A$ has the same block structure
but then $T$ is not
conjugate to $TA$ unless $A=I_4$.

Finally we assume that $T$ is one of the two 4 by 4 matrices
given in Proposition \ref{cent}, though we conjugate each one slightly
so that the centraliser takes on a neater form. This gives us\\
{\bf Case 1}:\\
\[T=\sma{cccc}\lambda&0&0&1\\0&\lambda&0&0\\0&0&\lambda&0\\
0&0&0&\lambda\fma,\mbox{ which has centraliser }
\sma{cccc}a&?&?&?\\0&?&?&?\\0&?&?&?\\0&0&0&a\fma\\
\]
where ? denotes any complex number, not necessarily the same number on
each appearance, whereas repeated letters are equal to each other.\\
{\bf Case 2}:\\
\[T=\sma{cccc}\lambda&0&1&0\\0&\lambda&0&1\\0&0&\lambda&0\\
0&0&0&\lambda\fma,\mbox{ which has centraliser }
\sma{cccc}a&b&?&?\\c&d&?&?\\0&0&a&b\\0&0&c&d\fma.\\
\]  
Moreover we can conjugate just in the middle 2 by 2 block in the first case,
and by the same upper left and lower right 2 by 2 blocks in the second case
which will leave $T$ unchanged but allow us to assume that $A$ is also upper
triangular. We now investigate each case.\\
{\bf Case 1}:\\
Setting 
\[A=\sma{cccc}1&a&b&c\\0&1&d&e\\0&0&1&f\\0&0&0&1\fma
\mbox{ gives }
TA=\sma{cccc}\lambda&\lambda a&\lambda b&\lambda c+1\\
0&\lambda&\lambda d&\lambda e\\0&0&\lambda&\lambda f\\
0&0&0&\lambda\fma
\]
and requiring $T$ and $TA$ to be conjugate forces $TA-\lambda I_4$
to have rank 1, thus if $f\neq 0$ then $a=b=d=0$. On taking
an arbitrary matrix $B=(b_{ij})$ and equating $BTA=TB$, we find
$b_{41}=b_{42}=b_{43}=0$ which means we can reduce to the 3 by 3
case by taking the top left 3 by 3 block in $T,A,B$, whereupon
we now have $A=I_3$ in this block. Consequently there is no
$B$ such that $A$ and $B$ generate a non abelian free group, for
instance $BAB^{-1}$ and $A$ are now both upper triangular 4 by 4
matrices.

We can repeat this argument if $a\neq 0$, so now say $a=f=0$. Here
we can work in 2 by 2 blocks to find out the conditions saying that
$TA-\lambda I_4$ and $TA^2-\lambda I_4$ have rank 1, which become
$be\lambda=d(c\lambda +1)$ and $2be\lambda=d(2c\lambda+1)$ respectively.
Together these imply that $d=0$ and then $b$ or $e=0$ which reduces
to the case above with the 3 by 3 block.\\
{\bf Case 2}:\\
Now $T-\lambda I_4$ has rank 2 and we will set
\[A=\sma{cccc}1&a&\alpha&\beta\\0&1&\gamma&\delta\\0&0&1&a\\0&0&0&1\fma
\mbox{ thus }
TA=\sma{cccc}\lambda&\lambda a&\lambda \alpha+1&\lambda\beta+a\\
0&\lambda&\lambda\gamma&\lambda\delta+1\\0&0&\lambda&\lambda a\\
0&0&0&\lambda\fma.
\]
We first consider when $a\neq 0$ whereupon the rank 2 condition
for $TA-\lambda I_4$ forces $\gamma=0$. However we also have
$(T-\lambda I_4)^2=0$ so we require $(TA-\lambda I_4)^2=0$ too,
which then implies that $\lambda a(\lambda\alpha+\lambda\delta+2)=0$,
or equivalently the (1,3) entry and the (2,4) entry
of $TA$ sum to zero, as $\lambda,a\neq 0$.

But if there also were a matrix $C$ with
$CTA^2=TC$ then the above argument implies that the (1,3) entry and
(2,4) entry of $TA^2$ also sum to zero, giving us
$2\lambda\alpha+2\lambda\delta+2=0$ and these cannot simultaneously
hold.

Finally we take $a=0$ in Case 2, where we now can write our 4 by 4
matrices in 2 by 2 blocks, thus
\[T=\sma{cc}\lambda I_2&I_2\\0&\lambda I_2\fma,
A=\sma{cc}I_2&X\\0&I_2\fma\mbox{ so }
TA=\sma{cc}\lambda I_2&\lambda X+I_2\\0&\lambda I_2\fma.\]
On again looking for a conjugating matrix $B=\sma{cc}P&Q\\R&S\fma$ and setting
$B(TA-\lambda I_4)=(T-\lambda I_4)B$, we force $R=0$. As the inverse of
$B$ will now be $\sma{cc}P^{-1}&-P^{-1}QS^{-1}\\0&S^{-1}\fma$, we again
see that $A$ and $BAB^{-1}$ are upper unitriangular and so certainly
cannot generate a non abelian free group. 
\end{proof}

We thus leave with two related questions:\\
{\bf Question 1}: Is every free by cyclic group $F_n\rtimes_\alpha\Z$
linear over $\C$?\\
{\bf Question 2}: Is there $N>0$ such that every free by cyclic group 
$F_n\rtimes_\alpha\Z$ embeds in $GL(N,\C)$?\\
It might be that the Gersten group, and indeed all free by cyclic groups
are linear though it looks as if more and more complicated representations
would be required to achieve this. As we can form variations on the
Gersten group, for instance if $F_n$ is free on $x_1,\ldots ,x_n$ then
take the automorphism $\alpha(x_1)=x_1,\alpha(x_2)=x_2x_1,\ldots,
\alpha(x_n)=x_nx_1^{n-1}$, we think it unlikely that Question 2 has a
positive answer, even if Question 1 does.

\Address


\begin{thebibliography}{99}

\bibitem{babnn} L.\,Babai, N.\,Nikolov and L.\,Pyber,
Product growth and mixing in finite groups. {\it Proceedings of the
Nineteenth Annual ACM-SIAM Symposium on Discrete Algorithms},
248--257, ACM, New York, 2008.

\bibitem{barbry} V.\,G.\,Bardakov and O.\,V.\,Bryukhanov,
{\it On linear representations of some extensions},
Vestnik Novosibirsk Univ. Ser. Mat. Mekh. Inf. {\bf 7} (2007)
45--58 (Russian).

\bibitem{bige}
S.\,J.\,Bigelow,
{\it The Burau representation is not faithful for $n=5$},
Geom. Topol. {\bf 3} (1999) 397-–404. 

\bibitem{bigel}  S.\,J.\,Bigelow, 
{\it Braid groups are linear}, 
J. Amer. Math. Soc. {\bf 14} (2001) 471--486.

\bibitem{bgbd} S.\,J.\,Bigelow and R.\,D.\,Budney,
{\it The mapping class group of a genus two surface is linear},
Algebr. Geom. Topol. {\bf 1} (2001) 699-–708.

\bibitem{me3} J.\,O.\,Button,
{\it A 3-manifold group which is not four dimensional linear},
J. Pure Appl. Algebra {\bf 218} (2014) 1604--1619.

\bibitem{crdk} M.\,Casals-Ruiz, A.\,Duncan and I.\,Kazachkov,
{\it embeddings between partially commutative groups: Two counterexamples},
J. Algebra {\bf 390} (2013) 87--99.

\bibitem{cl} D.\,Cooper and D.\,D.\,Long,
{\it A presentation for the image of 
$\mbox{Burau}(4)\otimes \Z_2$}, 
Invent. Math. {\bf 127} (1997) 535-–570.
 
\bibitem{clep} D.\,Cooper and D.\,D.\,Long,
{\it On the Burau representation modulo a small prime}, 
The Epstein birthday schrift, 127–-138, Geom. Topol. Monogr., 
{\bf 1}, Geom. Topol. Publ., Coventry, 1998.

\bibitem{drms} C.\,Droms,
{\it Graph groups, coherence and three-manifolds},
J. Algebra {\bf 106} (1987) 484--489.

\bibitem{ds} C.\,Dru\c tu and M.\,Sapir,
{\it Non-linear residually finite groups},
J. Algebra {\bf 284} (2005) 174--178.

\bibitem{dfg}  J.\,L.\,Dyer, E.\,Formanek and E.\,K.\,Grossman, 
{\it On the linearity of automorphism groups of free groups},
Arch. Math. (Basel) {\bf 38} (1982) 404–409. 

\bibitem{flm} B.\,Farb, A.\,Lubotzky and Y.\,Minsky,
{\it Rank-1 phenomena for mapping class groups},
Duke Math. J. {\bf 106} (2001) 581–-597.

\bibitem{flr}  B.\,Fine, Benjamin, F.\,Levin, Frank and
G.\,Rosenberger,
{\it Faithful complex representations of one relator groups},
New Zealand J. Math. {\bf 26} (1997) 45-–52.  

\bibitem{for} E.\,Formanek,
{\it Braid group representations of low degree},
Proc. London Math. Soc. {\bf 73} (1996) 279--322.

\bibitem{fp} E.\,Formanek and C.\,Procesi,
{\it The automorphism group of a free group is not linear},
J. Algebra {\bf 149} (1992) 494--499.

\bibitem{ger} S.\,M.\,Gersten,
{\it The automorphism group of a free group is not a CAT(0) group},
Proc. Amer. Math. Soc. {\bf 121} (1994) 999--1002.

\bibitem{hgws2} M.\,F.\,Hagen and D.\,T.\,Wise,
{\it Cubulating hyperbolic free-by-cyclic groups: the general case},
Geom. Funct. Anal. {\bf 25} (2015) 134--179.

\bibitem{hsuw} T.\,Hsu and D.\,T.\,Wise,
{\it On linear and residual properties of graph groups},
Michigan Math. J. {\bf 46} (1999) 251--259.

\bibitem{kam} M.\,Kambites,
{\it On commuting elements and embeddings of graph groups and monoids},
Proc. Edinb. Math. Soc. {\bf 52} (2009) 155-–170. 

\bibitem{kim} Kim, Sang-hyun, 
{\it Co-contractions of graphs and right-angled Artin groups},
Algebr. Geom. Topol. {\bf 8} (2008) 849–-868.

\bibitem{kmkb}  Kim, Sang-hyun and T.\,Koberda,
{\it Embedability between right-angled Artin groups},
Geom. Topol. {\bf 17} (2013) 493--530.

\bibitem{kork} M.\,Korkmaz, 
{\it On the linearity of certain mapping class groups}, 
Turkish J. Math. {\bf 24} (2000) 367-–371.

\bibitem{kou} The Kourovka notebook. Unsolved problems in group theory. 
Eighteenth edition. Edited by V. D. Mazurov and E. I. Khukhro. 
Russian Academy of Sciences Siberian Division, Institute of Mathematics, 
Novosibirsk, 2014. 

\bibitem{kra} D.\,Krammer,
{\it Braid groups are linear},
Ann. of Math. {\bf 155} (2002) 131-–156.

\bibitem{leesong} S.\,J.\,Lee and W.\,T.\,Song,
{\it The kernel of 
$\mbox{Burau}(4)\otimes \Z_p$ is all pseudo-Anosov}, 
Pacific J. Math. {\bf 219} (2005) 303-–310.

\bibitem{lop} D.\,D.\,Long and M.\,Paton, M,
{\it The Burau representation is not faithful for $n\geq 6$},
Topology {\bf 32} (1993) 439–-447. 

\bibitem{lr} D.\,D.\,Long and A.\,W.\,Reid, 
{\it Small subgroups of SL(3,$\Z$)}, 
Exp. Math. {\bf 20} (2011) 412-–425. 

\bibitem{lubseg} A.\,Lubotzky and D.\,Segal, Subgroup growth.
Progress in Mathematics 212, Birkha\"user Verlag, Basel, 2003.

\bibitem{magpkn} W.\,Magnus and A.\,Peluso,
{\it On knot groups},
Comm. Pure Appl. Math. {\bf 20} (1967) 749–-770.

\bibitem{mgtr}  W.\,Magnus and C.\,Tretkoff,
{\it Representations of automorphism groups of free groups}. 
Word problems, II (Conf. on Decision Problems in Algebra, Oxford, 1976), 
pp. 255–-259, Stud. Logic Foundations Math., 95, North-Holland, 
Amsterdam-New York, 1980. 

\bibitem{moopams} J.\,Moody, 
{\it The faithfulness question for the Burau representation},
Proc. Amer. Math. Soc. {\bf 119} (1993) 671-–679.  

\bibitem{nbws} G.\,A.\,Niblo and D.\,T.\,Wise,
{\it Subgroup separability, knot groups and graph manifolds},
Proc. Amer. Math. Soc. {\bf 129} (2000) 685--693.

\bibitem{nis} V.\,L.\,Nisnevi\u{c},
{\it \"Uber Gruppen, die durch Matrizen \"uber einem kommutativen Feld
isomorph darstellbar sind}, Rec. Math. [Mat. Sbornik] N.S. {\bf 8 (50)}
(1940), 395--403.

\bibitem{rom} V.\,A.\,Roman'kov,
{\it The linearity problem for the unitriangular automorphism groups
of free groups},
J. Siberian Federal Univ. Math. Phys. {\bf 6} (2013) 516--520.

\bibitem{sha} P.\,B.\,Shalen, {\it Linear representations of certain
amalgamated products}, J. Pure. Appl. Algebra {\bf 15} (1979), 187--197.

\bibitem{tenek} V.\,V.\,Tenekedzhi,
{\it Matrix representations of degree 6 of automorphism groups of free groups} 
(Russian), Some problems in differential equations and discrete mathematics 
(Russian), 109–-125, Novosibirsk. Gos. Univ., Novosibirsk, 1986. 

\bibitem{weh} B.\,A.\,F.\,Wehrfritz, {\it Generalized free products of
linear groups}, Proc. London Math. Soc. {\bf 27} (1973), 402--424.

\bibitem{wangagtop07} S.\,Wang,
{\it Representations of surface groups and 
right-angled Artin groups in higher rank},
Algebr. Geom. Topol. {\bf 7} (2007) 1099–-1117. 

\bibitem{wehbk} B.\,A.\,F.\,Wehrfritz:
Infinite linear groups. Ergebnisse der Matematik und ihrer Grenzgebiete,
Band 76. Springer-Verlag, New York-Heidelberg (1973).



\end{thebibliography}
\end{document}